\documentclass{amsart}

\usepackage[letterpaper,top=2cm,bottom=2cm,left=3cm,right=3cm]{geometry}

\usepackage[T1]{fontenc}
\usepackage[utf8]{inputenc}
\usepackage{lmodern}
\usepackage{hyperref}
\usepackage{float}
\usepackage{amsmath}
\usepackage{amsmath}
\usepackage{amsfonts}
\usepackage{amssymb}
\usepackage{amsthm}
\usepackage{graphicx}
\usepackage{xcolor}
\usepackage{subcaption}
\usepackage{todonotes}
\usepackage{cleveref}
\usepackage{mathtools}
\usepackage{colonequals}
\usepackage{booktabs}

\usepackage[style=alphabetic,doi=false,isbn=false,url=false,eprint=true,
backend=biber,giveninits=true,maxcitenames=3,maxbibnames=8,hyperref]{biblatex}
\theoremstyle{plain}
\newtheorem{thm}{Theorem}[section]
\newtheorem{theorem}[thm]{Theorem}
\newtheorem{lem}[thm]{Lemma}
\newtheorem{lemma}[thm]{Lemma}
\newtheorem{cor}[thm]{Corollary}
\newtheorem{corollary}[thm]{Corollary}

\newtheorem{proposition}[thm]{Proposition}

\newtheorem{conjecture}[thm]{Conjecture}
\newtheorem{notation}[thm]{Notation}
\newtheorem{thmx}{Theorem}

\theoremstyle{definition}
\newtheorem{df}[thm]{Definition}
\newtheorem{definition}[thm]{Definition}
\newtheorem{ex}[thm]{Example}
\newtheorem{rem}[thm]{Remark}
\newtheorem{remark}[thm]{Remark}

\newcommand{\R}{\mathbb{R}}
\newcommand{\Z}{\mathbb{Z}}
\newcommand{\B}{\ensuremath{\mathcal{B}}}
\newcommand{\D}{\ensuremath{\mathcal{D}}}

\newcommand{\A}{\ensuremath{\mathcal{A}}}
\newcommand{\Lat}{\ensuremath{\mathcal{L}}}
\newcommand{\C}{\ensuremath{\mathcal{C}}}

\newcommand{\precdot}{\prec\mathrel{\mkern-5mu}\mathrel{\cdot}}
\newcommand{\latDn}{\ensuremath{\mathcal{L}(\mathcal{D}_n)}}
\newcommand{\latBn}{\ensuremath{\mathcal{L}(\mathcal{B}_n)}}
\newcommand{\latDns}{\ensuremath{\mathcal{L}(\D_{n,s})}}
\newcommand{\latDnu}{\ensuremath{\mathcal{L}(\D_{n,u})}}

\newcommand{\edge}{\ensuremath{\mathcal{E}_{xy}}}
\newcommand{\indeg}{\text{indeg}}

\newcommand{\hf}{\text{hf}}
\renewcommand{\L}{\ensuremath{\mathcal{L}}}
\newcommand{\rk}{\text{rk}}
\newcommand{\Latt}{\text{Latt}}
\newcommand{\sm}{\hspace{-2pt} \shortminus \hspace{-3pt}}

\DeclareMathSymbol{\shortminus}{\mathbin}{AMSa}{"39}

\newcommand{\ra}[1]{\renewcommand{\arraystretch}{#1}}

\usepackage[foot]{amsaddr}

\title{On gamma-vectors and Chow polynomials of restrictions of reflection arrangements}

\author[S. Degen]{Sebastian Degen}
    \address{$^{1}$\textnormal{Universit\"{a}t Bielefeld, Fakult\"{a}t f\"{u}r Mathematik, Bielefeld, Germany.}}
    \email{sdegen@math.uni-bielefeld.de}

\author[L. Henetmayr]{Lisa Henetmayr$^1$}
    \email{lisa.henetmayr@uni-bielefeld.de}

\author[M. Mišinová]{Magdaléna Mišinová$^2$}
    \address{$^2$\textnormal{University of Konstanz, Department of Mathematics and Statistics, Konstanz, Germany.}}
    \email{magdalena.misinova@gmail.com}
    
\author[P. Pielasa]{Paweł Pielasa$^3$}
    \address{$^2$\textnormal{University of Cambridge, Department of Pure Mathematics and Mathematical Statistics, Cambridge, UK.}}
    \email{pp554@cam.ac.uk}
    
\author[F. Rieg]{Florian Rieg$^4$}
    \address{$^4$\textnormal{Goethe-Universität, Institut für Mathematik, Frankfurt am Main, Germany.}}
    \email{rieg@math.uni-frankfurt.de}

\begin{document}

\begin{abstract}
    Simplicial arrangements are a special class of hyperplane arrangements, having the property that every chamber is a simplicial cone. It is known that the simpliciality property is preserved under taking restrictions. In this article we focus on the class of reflection arrangements and investigate two different polynomial invariants associated to them and their restrictions, the $h$-polynomial with its $\gamma$-vector and the Chow polynomial. We prove that all restrictions of reflection arrangements are $\gamma$-positive and give an explicit combinatorial formula of the Chow polynomial in type $B$. Furthermore we prove that for a special class of restrictions of arrangements of type $D$, called intermediate arrangements, both the $h$-polynomial as well as the Chow polynomial behave arithmetically, that is they interpolate linearly between the respective invariants for type $B$ and $D$.\\
    
\end{abstract}
\maketitle

\section{Introduction}\label{sec: intro}

In the study of hyperplanes arrangements, the simplicial ones constitute a special class, having the property that every chamber is a simplicial cone. The most prominent examples, called reflection arrangements, arise from finite reflection groups, i.e., finite groups that are generated by reflections, see \cite{OT13}. Beyond reflection arrangements the simplicial ones are rare. In rank three, two infinite families and 95 sporadic examples of simplicial arrangements up to projective equivalence
are described in the Grünbaum--Cuntz catalog, see \cite{G09,C12,C22}, which is conjectured to be complete up to combinatorial isomorphism.

It is known that simpliciality is preserved under taking restrictions. An active field of research is concerned with investigating properties of restrictions of simplicial arrangements, see for instance \cite{MS22,C22}. For reflection arrangements, these restrictions are typically not reflection arrangements anymore. A full overview of the restrictions for all possible reflection arrangements and their relations is illustrated in \cite[Figure 9]{MS22}.

One special family arises from restrictions of reflection arrangements in type $D$. In this article we mainly focus on this family of arrangements often called \textit{intermediate arrangements}. 

We define these arrangements in $\R^n$ by means of their normal vectors
\[
\D_{n,s}:=\{e_i\pm e_j \colon 1\leq i< j\leq n\}\cup \{e_i \colon 1\leq i\leq s\},
\]
where $e_i$ is the $i$-th standard basis vector in $\R^n$ and $0\leq s\leq n$.
The name comes from the fact that $\D_{n,0}=\D_n$ and $\D_{n,n}=\B_n$ and so they can be viewed as interpolating between the reflection arrangements of type  $\D_n$ and $\B_n$, see \cite[Section 6.9]{OT13} for more details. Note that the arrangement $\D_{n,s}$ is denoted $\A_n^s$ in \cite{OT13}.

In this article we discuss two different polynomial invariants attached to the intermediate arrangements $\D_{n,s}$ and their respective properties. The first being the $h$-polynomial and its $\gamma$-vector and the second being the Chow-polynomial. We prove that both invariants form an arithmetic sequence with respect to $s$, meaning that the difference of the invariants associated to $\D_{n,s}$ and $\D_{n,s+1}$ is independent of $s$. In the proof, we show that such behavior can be seen more generally, in what we call union-additive invariants on coherent partition of a lattice. We give examples of other union-additive invariants. However, it is not clear what is the set of union-additive invariants and whether there are any other hyperplane arrangements satisfying the definition of a coherent partition.

\subsection{The $\gamma$-vector}

The first invariant is motivated by the fact that every simplicial arrangement $\A$ gives rise to a simplicial complex $\Delta_\A$, triangulating a sphere. Therefore one can define the $f$- and $h$-polynomial of the arrangement in the usual way:

Given the simplicial complex $\Delta_\A$ of dimension $d-1$, the $f$-vector $f(\A)=(f_{-1},f_0,\dots,f_{d-1})$ consists of the entries $f_k:=f_k(\A)$ recording the number of $k$-dimensional simplicies of $\A$, for $-1\leq k\leq d-1$. We define the \textit{$f$-polynomial} $f(t)\in\Z[t]$ as
\[
    f(t):=\sum_{i=0}^{d}f_{d-1-i} t^i
\]
and the \textit{$h$-polynomial}
\[
h(t):=f(t-1)=\sum_{i=0}^{d}h_i t^{i}
\]
whose coefficients form the \textit{$h$-vector} $h=(h_0,\dots,h_d)$. 

We known that in the case of a simplicial complex triangulating a sphere the $h$-vector is non-negative and symmetric, i.e.,  $h_i\geq 0$ and $h_i=h_{d-i}$ for all $i$. The latter are called the Dehn-Sommerville equations, see \cite{K64,G05}.
Therefore one can rewrite the $h$-polynomial in a different basis (\textit{$\gamma$-basis}) and obtain:
\[
h(t)=\sum_{i=0}^{\lfloor d/2\rfloor}\gamma_i \cdot t^i(1+t)^{d-2i}.
\]
The above coefficients $\gamma_i\in\Z$ form the so-called \textit{$\gamma$-vector} $\gamma=(\gamma_0,\dots,\gamma_{\lfloor d/2\rfloor})$ of the underlying simplicial complex $\Delta_\A$. We say that $\Delta_\A$ is $\gamma$-positive if $\gamma_i\geq 0$ for all $i$, see \cite{G05,A17} for more details.
The property of being $\gamma$-positive is widely studied, see \cite{A17} for an overview. The main motivation is drawn from the celebrated Gal's conjecture:

\begin{conjecture}\cite[Conjecture 2.1.7]{G05}
   If a simplicial complex is a flag sphere, i.e., the minimal non-faces are of size two and it triangulates a sphere, then it is $\gamma$-positive.
\end{conjecture}

Gal proved this conjecture for all flag spheres up to dimension 4, see\cite{G05}. It is known that simplicial arrangements yield flag spheres, so it follows that all simplicial arrangements of rank $\leq 5$ are $\gamma$-positive. Moreover, $\gamma$-positivity was proven for all reflection arrangements in \cite{S07}. 

One of our main results in this article confirms Gal's conjecture for all restrictions of reflection arrangements.
This is achieved by proving arithmeticity of the $\gamma$-coefficients for fixed $n$ and varying $0\leq s\leq n$ for the intermediate arrangements $\D_{n,s}$, together with \verb|SageMath|-computations in the exceptional cases.

\begin{thmx}
    \label{thm: intro_gam_pos}
    All restrictions of reflection arrangements are $\gamma$-positive.
\end{thmx}

As a byproduct of the proof of this theorem we then also obtain the exact increment of the arithmetic sequence in the cases of the intermediate arrangements, see \Cref{prop:explicit-increment}.

\subsection{The Chow polynomial}

The second invariant is motivated by an important geometric-algebraic object one can attach to a simple matroid $M$ on the groundset $E$, called the \textit{Chow ring} of $M$ and is defined as follows:

Let $\mathcal{L}(M)$ be the lattice of flats of $M$, then
\[
\text{CH}(M):=\frac{\mathbb{Q}[x_F\;|\;F\in\mathcal{L}(M)\setminus\{\emptyset\}]}{I+J},
\]
where $I,J$ are the two ideals
\begin{align*}
I &= \langle x_{F_1}x_{F_2}\;|\;F_1,F_2\in\mathcal{L}(M)\setminus\{\emptyset,E\}\text{ are incomparable}\rangle\\
J &= \Bigg\langle\sum_{F\ni i}x_F\;|\; i\in E\Bigg\rangle.
\end{align*}

Since their introduction in \cite{FY04}, they have attracted a lot of attention in the last years and are at the core of much current research in algebraic and geometric combinatorics, maybe most notably as the crucial tool in the resolution of the long-standing Heron–Rota–Welsh conjectures by Adiprasito–Huh–Katz in \cite{AHK18}. There, the authors showed that the Chow ring of a matroid satisfies the Kähler package, i.e., Poincaré duality, the Hodge-Riemann relations and the hard Lefschetz property. Thereby, they initiated the study of combinatorial Hodge theory.

The Hilbert-Poincaré series associated to $\text{CH}(M)$ is the polynomial defined as
\[
H_M(t)=\sum_{k\geq 0} \dim(\text{CH}(M)_k)t^k,
\]
where $\text{CH}(M)_k$ is the $k$-th graded piece of $\text{CH}(M)$. We call it the \textit{Chow polynomial} of $M$, see \cite{FMSV24}.

The Chow polynomial has been extensively studied in recent years. It is known to have palindromic coefficients due to Poincaré duality and conjectured to be real-rooted by June Huh and Matthew Stevens, see \cite[Conjecture 4.1.3 and 4.3.3]{S21}. In the uniform case this conjecture was recently confirmed by Brändén and Vecchi \cite{BV25}. 
Moreover Stump in \cite{S24} obtained a general combinatorial description of the Chow polynomial using an $R$-labeling for the lattice of flats of the matroid. He thereby reproved that the Chow polynomial has a $\gamma$-positive expansion, initially proven in \cite{FMSV24}. Furthermore he obtained a compact expression of this polynomial in the braid arrangement case (type $A$). Soon after Hoster gave an explicitly combinatorial formula, based on Schubert matroids, for the Chow polynomial of uniform matroids, see \cite{H24}.

For any sequence $(a_1,\ldots,a_n)\in\mathbb{Z}^n$, we denote by $\text{des}(a_1,\ldots,a_n)$ the number of descents of the sequence, i.e. the number of indices $i$ such that $a_i>a_{i+1}$.
In this article we provide an explicit combinatorial formula of the Chow polynomial for the reflection arrangement of type $B$, by using type $B$ analogues of the methods in \cite{S24}.

\begin{thmx}
\label{thm: intro_chow_poly_Bn}
The Chow polynomial for type $B$ has the expansion
\[
    H_{\B_n}(x)=\sum_{(a_1,\dots,a_n)}\left(\prod_{i=1}^n (2a_i-1)\right)x^{des(a_1,\dots,a_n)}(x+1)^{n-1-2des(a_1,\dots,a_n)}
\]
where the sum ranges over all tuples $(a_1,\dots,a_n)$ with $a_i\in \{1,\dots,n+1-i\}$ such that $a_1\leq a_2$ and also that $a_i>a_{i+1}$ implies $a_{i-1}\leq a_{i}$ for $i\in\{2,\dots,n-1\}$.
\end{thmx}

Moreover we show that the Chow polynomials for the intermediate arrangements $\D_{n,s}$ behaves arithmetically with respect to fixed $n$ and varying $0\leq s\leq n$. This is done in two different ways, using on the one hand an EL-labeling of the intersection lattices as introduced in \cite{DNG19}, see \Cref{prop:chow_arithmeticity_EL}, and on the other hand a recursive approach via reduced characteristic polynomials of the minors of the underlying matroids, see \Cref{thm: chow_arithmeticity_char_polyn}.

The rest of the article is organized as follows: In \Cref{sec:prelims} we briefly recall the basic notions of hyperplane arrangements and matroids, needed for the discussion in the later sections. In \Cref{sec: arithmeticity-of-matroid-invariants} we discuss certain class of invariants with connection to the intermediate arrangements $\D_{n,s}$, showing that these invariants behave arithmetically on $\D_{n,s}$ with respect to fixed $n$ and varying $0\leq s\leq n$. In \Cref{sec:gamma_pos_part} we apply these results to $f$-polynomial and subsequently show the $\gamma$-positivity for arrangements $\D_{n,s}$.
Moreover, we explicitly determine the increment of the above mentioned arithmetic sequence. Afterwards, we discuss a computational approach via \verb|SageMath| to verify the $\gamma$-positivity of all rank $6$ and $7$ restrictions of the exceptional cases, which then enables us to conclude the proof of \Cref{thm: intro_gam_pos}. In \Cref{sec:chow_polyn_part} we introduce an $R$-labeling for the lattice of flats of matroids arising from reflection arrangements of type $B$. We then apply \Cref{thm:chow_formula} to derive the explicit formula of their Chow polynomials, as stated in \Cref{thm: intro_chow_poly_Bn}. Thereafter we discuss two approaches for proving the arithmetical behavior of the Chow polynomials associated to the matroids underlying the $\D_{n,s}$-arrangements. For the first approach we examine a recursive formula of their Chow polynomials in terms of reduced characteristic polynomials of minors and apply result of section \Cref{sec: arithmeticity-of-matroid-invariants} again and for the second approach we investigate an EL-labeling for their lattice of flats. At the end of the article (\Cref{appendixsec: chow_polyn_catalogue}) we provide a small catalogue of Chow polynomials of $\D_{n,s}$-matroids up to dimension $n=7$.

\section*{Acknowledgments}

This paper stemmed from a project within a Dive into Research, which grants research experience to undergraduates. We thank Michael Cuntz, Lukas Kühne, and Raman Sanyal, who organized the Dive into Research. We are also thankful for the funding for the Dive into Research from the DFG priority program SPP 2458 Combinatorial Synergies. We want to give special thanks to Michael Cuntz, for providing us essential gamma-vector computations and Lukas Kühne and Raman Sanyal, who posed the starting questions and helped a lot with proofreading and insightful discussions.
Sebastian Degen was supported by the Deutsche Forschungsgemeinschaft (DFG, German Research Foundation) -- SFB-TRR 358/1 2023 -- 491392403 and SPP 2458 -- 539866293. Florian Rieg was supported by the Spanish-German project COMPOTE (DFG 541393733) as well as by the DFG priority program \emph{Combinatorial Synergies} (SPP 2458) — 539866293.

\section{Preliminaries}\label{sec:prelims}

\subsection{Notation}

Throughout the article for a given positive integer $n$ we set $[n]:=\{1,\ldots,n\}$ and denote be $S_n$ the symmetric group on $n$ elements.

\subsection{Hyperplane arrangements}\label{subsec:arrangements}

In this subsection we review some basics theory of hyperplane arrangements, that we use frequently throughout the article. For a more detailed overview the reader may consult \cite{OT13}.

\begin{definition}
    A \textbf{hyperplane in $\R^n$} is a $(n-1)$-dimensional linear subspace $H$, which can be describe in the following way
    \[
        H:=\{(x_1,\ldots,x_n)\in\R^n\;|\;\alpha^tx=\alpha_1x_1+\ldots+\alpha_nx_n=0\},
    \]
    for some $\alpha:=(\alpha_1,\ldots,\alpha_n)\in\R^n\setminus\{0\}$. We call the vector $\alpha$ a \textbf{normal of $H$} and denote the hyperplane defined by $\alpha$ as $H_\alpha$.
\end{definition}

The closed and open positive half spaces corresponding to $H_\alpha$ are defined as
\[
    H^\geq_\alpha:=\{x\in\R^n\;|\;\alpha^t x\geq0\}\; \text{ and }\; H^>_\alpha:=\{x\in\R^n\;|\;\alpha^t x>0\}.  
\]
Analogously, we define the negative half spaces denoted by $H_\alpha^\leq$ and $H_\alpha^<$.

\begin{definition}
    A \textbf{hyperplane arrangement $\A$} in $\R^n$ is a finite non-empty collection of hyperplanes. We say \textit{arrangement} for short.
\end{definition}

Note that in this context, every hyperplane $H$ is linear, i.e. contains the origin. In other words, we only consider so called \textbf{central} hyperplane arrangements $\A$. Every arrangement comes with an associated lattice, encoding the intersections of the hyperplanes. 

\begin{definition}
    The \textbf{lattice of intersections} of an arrangement $\A=\{H_1,\ldots,H_m\}$ in $\R^n$ is the poset $(\Lat(\A), \preceq)$, where
    \[
        \Lat(\A):=\Bigg\{\bigcap_{i\in I}H_i\;|\;I\subseteq[m]\Bigg\}
    \]
    and $\leq$ denotes the reverse inclusion, i.e. for all $X,Y\in\Lat(\A)$ we have $X\preceq Y$ if $Y\subseteq X$. 
    In the following, we abbreviate this poset by $\Lat(\A)$.
\end{definition}

\begin{remark}
    For any arrangement $\A$, the poset $\Lat(\A)$ is a geometric lattice. The \textbf{rank} of an element $X\in\Lat(\A)$ is its codimension in $\R^n$ and we call the codimension of $\bigcap_{H\in\A}H$ the \textbf{rank} of the arrangement $\A$. Moreover, the \textbf{atoms} of the lattice, i.e., codimension $1$ elements, are the hyperplanes in $\A$.
\end{remark}

For an arrangement $\A$ in $\R^n$ the complement $\R^n\setminus\bigcup_{H\in\A}H$ is disconnected and the connected components are of special interest in the study of hyperplane arrangements:

\begin{definition}
    The \textbf{chambers} of $\A$ are the closures of the connected components of $\R^n\setminus\bigcup_{H\in\A}H$. We denote the set of all chambers with $\C(\A)$. An intersection of two chambers spanning a codimension-1 subspace is called a \textbf{wall} of both chambers. Every wall is contained in a unique hyperplane of $\A$, which is given by span of the aforementioned intersection. For a given wall we call this hyperplane the \textbf{supporting hyperplane} of the wall. 
\end{definition}

In this article, we mainly focused on a particular class of arrangements, having specially shaped chambers.

\begin{definition}
    Let $\A$ be an arrangement. We call a chamber $C\in\C(\A)$ \textbf{simplicial} if the normal vectors of all the supporting hyperplanes of all walls of $C$ are linearly independent. If all elements of $\C(\A)$ are simplicial, then the arrangement $\A$ is called \textbf{simplicial}. 
\end{definition}

The most prominent examples of simplicial arrangements arise from \textit{finite reflection groups}, i.e., finite groups that are generated by orthogonal reflections in hyperplanes in $\R^n$. The irreducible finite reflection groups can be classified into types, consisting of the infinite families $A_n, B_n, D_n, I_2(m)$ as well as the exceptional cases $E_6,E_7,E_8,F_4,H_3,H_4$; see \cite{Hum90}.
Let $W$ be a finite reflection group acting on $\R^n$, then the collection $\A(W)$ of reflecting hyperplanes associated to the reflections in $W$ is called the \textbf{reflection arrangement of $W$}. 

\begin{ex}
     The reflection arrangements in $\R^n$ corresponding to the reflection groups of type $A, B$ and $D$ are given by the following sets of normal vectors of the reflecting hyperplanes:
     \begin{itemize}
         \item $\A_{n-1}:=\{e_i-e_j\;|\;1\leq i<j\leq n\}$,
         \item $\B_n:=\{e_i\pm e_j\;|\;1\leq i<j\leq n\}\cup\{e_i\;|\;1\leq i\leq n\}$ and
         \item $\D_n:=\{e_i\pm e_j\;|\;1\leq i<j\leq n\}$.
     \end{itemize}
    Here $e_i$ denote the standard basis vectors of $\R^n$, for $1\leq i\leq n$.
\end{ex}

Lastly, we review some common operations on arrangements.

\begin{definition}\label{def: deletion_restriction_arr}
   Let $\A$  be an arrangement in $\R^n$. Any subset of $\A$ is called a \textbf{subarrangement} of $\A$. The \textbf{deletion} of a hyperplane $H\in\A$ is the subarrangement $\A\setminus\{H\}\subseteq\A$ in $\R^n$. The \textbf{restriction} to a hyperplane $H\in\A$ is an arrangement in $H\cong\R^{n-1}$ given by:
   \[
        \A^H:=\{H\cap H^\prime\;|\;H^\prime\in\A\setminus\{H\}\}.
   \]
\end{definition}

\begin{remark}
    If $\A$ is a simplicial arrangement, then its restriction to any hyperplane $H\in\A$ is a simplicial arrangement as well. In this way, we can construct new simplicial arrangements from already known ones. For reflection arrangements, a full overview of their restrictions and relations is illustrated in \cite[Figure 9]{MS22}.
\end{remark}

\subsection{Matroids}\label{subsec:matroids}

In this subsection we briefly recap some general matroid theoretical notions and concepts that are used in the discussion of Chow polynomials, see  \Cref{sec:chow_polyn_part}.
We refer to \cite{O2011} for an overview of matroid theory.

\begin{definition}\label{def: mat}
    A \textbf{matroid} $M$ is a pair $(E,r)$, consisting of a finite set $E$ and a function $r:2^E\rightarrow\Z$, called the \textbf{rank function} satisfying the following axioms for all $A,B\in 2^E$:
    \begin{enumerate}
        \item  $0\leq r(A)\leq |A|$,
        \item  If $A\subseteq B$, then $r(A)\leq r(B)$, and
        \item  $r(A\cap B)+r(A\cup B)\leq r(A)+r(B)$. 
    \end{enumerate}
\end{definition}

Let $M=(E,r)$ be a matroid. We call $r(M):=r(E)$ the \textbf{rank} of $M$. A subset $A\in 2^E$ is called \textbf{independent} if $r(A)=|A|$, otherwise it is called \textbf{dependent}. A dependent subset of size 1 is called \textbf{loop} and one of size $2$ with rank $1$ is called \textbf{parallel element}. If there are no loops and parallel elements, $M$ is a \textbf{simple} matroid. The relation on the groundset of a matroid $M$ of being parallel is an equivalence. We can thus define a \textbf{simplification} of $M$ by deleting all loops from $M$ and identifying parallel elements.
Finally, we call a subset $F\in 2^E$ a \textbf{flat} of $M$ if $r(F\cup e)>r(F)$ for all $e\in E\setminus F$.

Every matroid $M$ can be equivalently described by its collection of independent sets, dependent sets, flats and hyperplanes, see \cite{O2011}. These equivalent descriptions are often called \textbf{cryptomorphic characterizations} of $M$.

The collection of flats of a matroid admits a poset lattice structure.

\begin{definition}
    Let $M=(E,r)$ be a matroid. Then the collection of flats of $M$ forms a lattice with respect to inclusion, called the $\textbf{lattice of flats}$ of $M$. We denote this lattice by $\Lat(M)$. For two flats $F,G\in\Lat(M)$ the \textbf{join} and the \textbf{meet} are defined as
    \[
    F\wedge G:=F\cap G\quad \text{and}\quad F\vee G:=\text{cl}(F\cup G),
    \]
    where $\text{cl}:2^E\rightarrow2^E$ is a function called the \textbf{closure}, assigning every $A\subseteq E$ the set
    \[
        \text{cl}(A):=\{x\in E\;|\;r(A\cup x)=r(A)\}.
    \]
\end{definition}

Note that from the definition of the closure function we see that a subset $F\subseteq E$ is a flat of $M$ if $F$ is its own closure, i.e., $\text{cl}(F)=F$.

\begin{remark}\label{rem: sim_mat_geom_lat_corresp}
    One can show that the lattice of flats of matroid even constitutes a geometric lattice. Moreover, there is even a one-to-one correspondence between simple matroids and geometric lattices, see \cite[Section 1.7]{O2011} for a detailed discussion.
    
\end{remark}

Now we review an important class of matroids, which provides the bridge to hyperplane arrangement theory. For our discussion in \Cref{sec:chow_polyn_part} we only use these kind of matroids. 

\begin{definition}\label{def: mat_repable}
    Let $\mathbb{K}$ be field and $\A=\{H_e\}_{e\in E}$ be an arrangement in $\mathbb{K}^n$, for some finite set $E$. This yields a matroid $M(\A)=(E,r)$ defined via the rank function
    \[
        r(S):=n-\dim\Bigg(\bigcap_{e\in S}H_e\Bigg).
    \]
    We call matroids arising in this way \textbf{representable over $\mathbb{K}$}.
\end{definition}

Note that the lattice of flats of a matroid corresponding to an arrangement coincides with the lattice of intersection of the arrangement, see \cite[Section 6]{O2011}. This means that the matroid is simple, since there exists no loops or parallel elements.

There are some operations we can perform on a given matroid to obtain new matroids.

\begin{definition}
    Let $M=(E,r_M)$ be matroid. The \textbf{restriction of $M$ to a subset $S\subseteq E$} is the matroid $M|_S:=(S,r_{M|_S})$ with rank function
    \[
        r_{M|_S}:=r_M(A)\quad \text{for all } A\subseteq S.
    \]
    Restricting to $E\setminus S$ is often called the \textbf{deletion of $S$ from $M$}.
    
    The \textbf{contraction of $M$ by a subset $T\subseteq E$} is the matroid $M/T:=(E\setminus T, r_{M/T})$ with rank function
    \[
         r_{M/T}(A):=r_M(A\cup T)- r_M(A)\quad \text{for all } A\subseteq E\setminus T.
    \]
\end{definition}

\begin{rem}
    If a matroid arises from an arrangement as in \Cref{def: mat_repable} the deletion and restriction operation of the arrangement as in \Cref{def: deletion_restriction_arr} correspond to the deletion and contraction operation of the matroid respectively.
\end{rem}

Lastly we would like to review a polynomial invariant attached to a given matroid, which is used for the discussion later on in \Cref{subsec:arthmeticity_chow_EL}.

\begin{definition} \label{def-char}
    Let $M=(E,r)$ be a matroid of rank $k$. The \textbf{characteristic polynomial $\chi(M)$ of $M$} is defined as
    \[
        \chi(M)(t):=\sum_{S\subseteq E}(-1)^{|S|}t^{k-r(S)}\in\R[t].
    \]
\end{definition}

Note that $t=1$ is always a root of the characteristic polynomial $\chi(M)$. Therefore we also define the \textbf{reduced characteristic polynomial of $M$} given by $\overline{\chi}(M)(t):=\frac{\chi(M)(t)}{t-1}$.

\section{$\D_{n,s}$ arrangements and arithmeticity of matroid invariants} \label{sec: arithmeticity-of-matroid-invariants}

The intermediate arrangements $\D_{n,s}$ can be viewed as an interpolation between the arrangements of type $D$ and $B$. In the spirit of this, we explore matroid invariants $\phi$ of the $\D_{n,s}$-arrangements that interpolate linearly between $\phi(\D_n)$ and $\phi(\B_n)$. We exploit the fact that the sequence of the corresponding lattices of flats behaves in a very special way. That is, they induce what we will call a \emph{coherent} partition of the lattice of flats of $\B_n = \D_{n,n}$. The main result of this section is roughly that the elements of the incidence algebra of a lattice $\L$ that behave arithmetically on any coherent partition of $\L$ form a subalgebra. Since this subalgebra is further closed on multiplying by functions that depend only on the rank of the lattice, it follows that a lot of known matroid invariants are in it as well. Since we work only with functions depending only on the isomorphism type of the lattice, we will not further talk about incidence algebra, but about functions from lattices.

\subsection{Union-additive invariants of lattices}

Next, we introduce the notion of \emph{coherent partition} of a lattice and of a union-additive invariant on such a partition. We will see that these notions exactly fit the case of $\D_{n,s}$ arrangements and the invariants we are interested in.

\begin{notation}
    We denote by $\Latt$ the set of all finite graded lattices. For $\L\in \Latt$, we denote by $\rk(\L)$ the rank of the maximal element. Then $\Latt_{\leq r}\subset \Latt$ denotes the set of lattices of rank at most $r$. Further, we denote by $0$ the minimal element and by $1$ the maximal element.
\end{notation}
\begin{definition}
    Let $\L = (L,\leq)$ be a finite graded lattice. We call $(L_0, \dots ,L_k)$ a \emph{coherent partition of $\L$} if
    \begin{enumerate}
        \item $(L_0,\dots, L_k)$ is a partition of $L$, with $L_i$ potentially empty for some $i$,
        \item $\L_0:=(L_0,\leq)$, $\L_1:= (L_0\cup L_1,\leq)$, $\dots$, $\L_k:=(L_0\cup L_k, \leq)$ are finite graded lattices,
        \item $\rk(\L_0) = \rk(\L)$,
        \item (incomparability) if $i\neq j$, $i,j>0$, and $x\in L_i$, $y\in L_j$, $x$ and $y$ are incomparable in $\L$,
        \item (convexity) if $x,z\in L_i$ for some $i\geq 1$, $y\in L$ and $x\leq y\leq z$, then $y\in L_i$.
    \end{enumerate}
\end{definition}
Note that item (3) implies that the minimal and maximal element of $\L$ are in $\L_i$ for all $i\geq 0$. Also note that for most of the theory we can weaken the assumption that $\L_i$'s are finite graded lattices to them being finite graded posets with a unique maximal and minimal element.

\begin{definition}
    Let $R$ be a commutative ring. We call a function $f: \Latt\to R$ \emph{union-additive on $\Latt_{\leq r}$} if for every finite graded lattice $\L = (L,\leq)\in \Latt_{\leq r}$ and for every coherent partition $(L_0,\dots ,L_k)$ of $\L$ we have
    $$
        f(\L) = f(\L_0) + \sum_{i=1}^k (f(\L_i)-f(\L_0)).
    $$
    We say that $f$ is \emph{union-additive} if it is union-additive on $\Latt_{\leq r}$ for every $r\in \mathbb{N}_0$.
\end{definition}
Usually, we consider $R = \Z[t]$.

We are interested in coherent partitions and union additive functions because of the following two propositions. They show how union-additive invariants behave on certain family of hyperplane arrangements. We will abuse notation and write $\L(\mathcal A)$ also for the ground set of the lattice of flats of a hyperplane arrangement $\A$.

\begin{proposition}\label{prop:description-of-lattice}
    Let $\A$ be a hyperplane arrangement and let $H_1$, $\dots$, $H_k$ be hyperplanes not in $\A$ such that for all $i\neq j$ we have $H_i\cap H_j = \bigcap_{H\in S\subset \A} H$ for some set of hyperplanes $S \subset \A$. Then $(\L(\A), \L(\A\cup\lbrace H_1\rbrace)\setminus\L(\A), \dots, \L(\A\cup\lbrace H_k\rbrace)\setminus \L(A))$ is a coherent partition of $\L(\A\cup\lbrace H_1,\dots ,H_k)\rbrace$.
\end{proposition}
\begin{proof}
    The items (2) and (3) of the definition of a coherent partition are obviously satisfied.
    
    Observe that for all $r> 1$ and $\lbrace i_1,\dots ,i_r\rbrace\subset [k]$, 
    $$
        H_{i_1}\cap \dots \cap H_{i_r} = (H_{i_1}\cap H_{i_2})\cap \dots \cap (H_{i_{r-1}}\cap H_{i_r}),
    $$ 
    hence $H_{i_1}\cap \dots \cap H_{i_r} = \bigcap_{H\in S\subset \A} H$ for some $S \subset \A$.
    
    Item (1): Let $F = \bigcap_{i\in I\subset [k]}H_i \cap \bigcap_{H\in S\subset \A} H$ be a flat of $\A\cup\lbrace H_1,\dots ,H_k\rbrace$. If $|I| = 1$, say $I = \lbrace i \rbrace$, $F$ is a flat of $\A\cup \lbrace H_i\rbrace$. If $|I| = 0$, $F$ is a flat of $\A$. If $|I|\geq 2$, we can replace $\bigcap_{i\in I\subset [k]}H_i$ by $\bigcap_{H\in T\subset \A} H$ for some $T$, showing that $F$ is a flat of $\A$.

    Incomparability: Let $F$ and $G$ be flats in $\L(\A\cup\lbrace H_i\rbrace)\setminus\L(\A)$ and $\L(\A\cup\lbrace H_j\rbrace)\setminus\L(\A)$ respectively. Assume that $F\subset G$. We know that $F = H_i \cap \bigcap_{H\in S\subset A} H$ for some $S$ and $G = H_j \cap \bigcap_{H\in T\subset A} H$ for some $T$. In particular, $G\subset H_j$, so $F\subset H_j$ and we have $F = H_i\cap H_j\cap \bigcap_{H\in S\subset A} H$. But then we can replace $H_i\cap H_j$ by an intersection of hyperplanes in $\A$, contradicting that $F$ is not a flat of $\A$.

    Convexity: We proceed by contradiction. Take $F_1, F_3\in \L(A\cup \lbrace H_i\rbrace)$ such that $F_1,F_3\not\in \L(\A)$ and $F_2\in \L(A\cup\lbrace H_1,\dots , H_k)$ such that $F_1\subset F_2\subset F_3$ and for contradiction $F_2\not\in L_i$. By incomparability, it suffices to assume that $F_2\in \L(\A)$. We can write 
    \begin{align*}
        F_1 &= H_i \cap \bigcap_{H\in S\subset \A} H,
    \end{align*}
    for some $S$. Since $F_1\subset F_2\subset F_3 \subset H_i$, we can replace $H_i$ by $F_2$ in the expression for $F_1$. Because $F_2\in \L(\A)$, this shows that $F_1\in \L(\A)$, which is a contradiction.
\end{proof}

Recall that the $\mathcal D_n$ arrangement in $\R^n$ consists of hyperplanes $\lbrace x_i \pm x_j = 0\mid i,j\in [n],i\neq j\rbrace$. For $s\in [n]\cup\lbrace 0\rbrace$, we define the arrangement $\D_{n,s}$ as $\mathcal D_n$ with added hyperplanes $\lbrace x_i = 0\rbrace$ for $i\in [s]$. Then $\mathcal D_{n,n} = \mathcal B_n$. The following result shows that union-additive invariants behave arithmetically on the lattices of $\D_{n,s}$ arrangements. This motivates the subsequent development of tools for showing union-addititvity for many invariants.

\begin{proposition}\label{prop:arithmeticity-for-Dns}
    Let $R$ be a commutative ring and let $\phi:\Latt\to R$ be a union-additive function. Then for fixed $n$, $\phi(\L(\mathcal D_{n,s})) - \phi(\L(\mathcal D_{n,{s-1}}))$ is independent of the choice of $s\in [n]$.
\end{proposition}
\begin{proof}
    For $i\neq j$, 
    $$\lbrace x_i = 0\rbrace \cap \lbrace x_j=0\rbrace = \lbrace x_i-x_j = 0\rbrace \cap \lbrace x_i+x_j = 0\rbrace.$$
    Hence, the arrangement $\mathcal D_n$ with the hyperplanes $\lbrace x_1=0\rbrace$, $\dots$, $\lbrace x_s = 0\rbrace$ satisfies the assumptions of Proposition \ref{prop:description-of-lattice}, so
    $$
        \phi(\L(\mathcal D_{n,s})) = \phi(\L(\mathcal D_n)) + \sum_{i=1}^s \phi(\L(\mathcal D_n\cup\lbrace\lbrace x_i=0\rbrace\rbrace)) - \phi(\L(\mathcal D_n)).
    $$
    
    By symmetry, $\mathcal D_n\cup\lbrace\lbrace x_i=0\rbrace\rbrace$ is isomorphic over all $i$, hence 
    $$
    \phi(\L(\mathcal D_{n,s})) - \phi(\L(\mathcal D_{n,{s-1}})) = \phi(\L(\mathcal D_n\cup\lbrace\lbrace x_s=0\rbrace\rbrace)) - \phi(\L(\mathcal D_n))$$ 
    is independent of $s$.
\end{proof}

The following theorem summarizes how we can obtain union-additive functions.

\begin{theorem}\label{thm:subalg-of-union-additive-functions}
    Let $R$ be a commutative ring and let $\mathcal{R}$ be the ring of functions $\Latt \to R$ (equipped with pointwise addition and multiplication) that depend only on the rank of a lattice. Let $\mathcal S$ be the ring of all functions $\Latt \to R$ where the multiplication is the convolution, that is
    \[
        (f\star g)(\L) = \sum_{x\in \L} f([0,x]_\L)\cdot g([x,1]_\L) ~~~\text{for}~~~ f,g\in \mathcal S.
    \]
    Furthermore, we endow $\mathcal S$ with point-wise multiplication by elements of $\mathcal R$, that is 
    \[
        (r\cdot f)(\L) = r(\L)\cdot f(\L)~~~\text{for}~~~r\in \mathcal R, f\in \mathcal S.
    \]
    Then the union-additive functions in $\mathcal S$ form a subring closed under the point-wise multiplication by elements of $\mathcal R$.
\end{theorem}
\begin{proof}
    By item (3) of the definition, for any $r\in \mathcal{R}$ we have $r(\L)=r(\L_i)$ for all $i\geq 0$.
    This implies that elements of $\mathcal R$ are union-additive. By linearity, union-additive functions are closed under addition and under point-wise multiplication by elements in $\mathcal R$. It remains to show that they are closed under convolution, which is done in Proposition \ref{prop:sum-over-simple-products}.
\end{proof}

\begin{remark}
    Note that convolution is the product considered in the incidence algebra of a lattice.
\end{remark}

To prove \Cref{prop:sum-over-simple-products}, we need the following:

\begin{lemma}\label{lemma:intersecting-with-interval}
    Consider $\L\in \Latt$ and its coherent partition $(L_0,\dots ,L_k)$. Let $x,y\in L_0$. Then $(L_0\cap [x,y], \dots, L_k\cap [x,y])$ is a coherent partition of $[x,y]_\L$.
\end{lemma}
\begin{proof}
    We check the definition:
    \begin{enumerate}
        \item Intersecting subset with partition gives a partition of that subset again.
        \item Follows from $x,y\in L_0$.
        \item Finite graded lattices are closed under taking intervals.
        \item Taking smaller lattice does not change this.
        \item Same as the previous point.
    \end{enumerate}
\end{proof}

\begin{proposition}(Convolution)\label{prop:sum-over-simple-products}
    Let $R$ be a commutative ring, $f_1,f_2,f:\Latt\to R$, $f_1$ and $f_2$ union-additive, such that for every $\L = (L,\leq )\in \Latt$
    $$
        f(\L) = \sum_{x\in L} f_1([0,x]_\L)f_2([x,1]_\L).
    $$
    Then $f$ is union-additive.
\end{proposition}
\begin{proof}
    For any lattice $\L$ and $x\in \L$, we denote $f_1([0,x]_\L)f_2([x,1]_\L)$ by $F_{x,\L}$.
    
    Fix any lattice $\L$ and let $(L_0,\dots, L_k)$ be a coherent partition of $\L$. Since this is a partition of $L$, we can split the sum as follows:
    $$
    \sum_{x\in L} F_{x,\L} = \sum_{i=0}^k\sum_{x\in L_i} F_{x,\L}.
    $$
    By the incomparability property, $[0,x]_\L =[0,x]_{\L_i}$ and $[x,1]_\L = [x,1]_{\L_i}$ whenever $x\in L_i$ for some $i\geq 1$. Hence, in this case, $F_{x,\L} = F_{x,\L_i}$.

    If $x\in L_0$, incomparability and convexity imply that at most one of $[0,x]_\L$ and $[x,1]_\L$ contains an element not in $L_0$: Picking some $y < x$ and $z >x$, it cannot be that $y,z$ are both in $L_i$ for some $i$ by convexity, and there cannot be $i\neq j$ such that $y\in L_i$ and $z\in L_j$ because otherwise $y$ and $z$ would be comparable, a contradiction. Then there are only three cases to consider:
    \begin{align*}
        F_{x,\L} &=F_{x,\L_0} = F_{x,\L_0} + \sum_{i=1}^{k} \left(F_{x,\L_i} - F_{x,\L_0}\right)& \text{if } [0,x]_\L,[x,1]_\L\subset \L_0,\\
        F_{x,\L} &= f_1([0,x]_\L)f_2([x,1]_\L) = \\
        & = f_2([x,1]_{\L_0})\cdot \left( f_1([0,x]_{\L_0}) + \sum_{i=1}^k f_1([0,x]_{\L_i}) - f_1([0,x]_{\L_0})\right)=\\
        &= F_{x,\L_0} + \sum_{i=1}^{k} \left( F_{x,\L_i} -  F_{x,\L_0}\right)
        &\text{if }[0,x]_{\L}\not\subset \L_0,\\
         F_{x,\L}&= F_{x,\L_0} + \sum_{i=1}^{k} \left( F_{x,\L_i} -  F_{x,\L_0}\right) &\text{if }[x,1]_{\L}\not\subset \L_0,
    \end{align*}
    where  in the second step of the first equality we just added zero, in the proof of the second equality we use union-additivity of $f_1$ and Lemma \ref{lemma:intersecting-with-interval} for the interval $[0,x]_\L$, and the last equality is proved analogously as the second one. Altogether,
    \begin{align*}
        \sum_{i=0}^k\sum_{x\in L_i} F_{x,\L} &= \sum_{x\in L_0} \left(  F_{x,\L_0} + \sum_{i=1}^{k} \left( F_{x,\L_i} -  F_{x,\L_0}\right)\right) + \sum_{i=1}^k \sum_{x\in L_i}  F_{x,\L_i} =\\
        &= \sum_{x\in L_0}  F_{x,\L_0} + \sum_{i=1}^k \left(\sum_{x\in L_0\cup L_i}  F_{x,\L_i} - \sum_{x\in L_0} F_{x,\L_0}\right) = \\
        & = f(\L_0) + \sum_{i=1}^k\left( f(\L_i) -f(\L_0)\right),
    \end{align*}
    as desired.
\end{proof}

The following propositions can be used to prove union-additivity by induction.

\begin{definition}
    Let $\delta_{\rk = r}$ be the function which is $1$ on lattices of rank $r$ and zero otherwise. We define $\delta_{\rk \neq r}$ analogously.
\end{definition}

\begin{proposition}(Recursion from above)\label{prop:recursion-from-above}
    Let $R$ be a commutative ring, $r\in \mathbb N_0$, $f_1,f_2,f:\Latt\to R$, $f_1$ union-additive, $f_2$ union-additive on $\Latt_{\leq r}$, such that for every $\L\in \Latt_{\leq r+1}$
    $$
        f(\L) = \sum_{x\in \L, x\neq 0} f_1([0,x]_\L)f_2([x,1]_\L).
    $$
    Then $f$ is union-additive on $\Latt_{\leq r+1}$.
\end{proposition}
\begin{proof}
    We can extend $f_2$ to $f_2':\Latt\to R$ by $f_2'(\L) = f_2(\L)$ if $\rk(\L) \leq r$ and $f_2'(\L) = 0$ otherwise. Then $f_2'$ is union-additive. Furthermore, let $f_1':\Latt \to R$ be the pointwise product $\delta_{\rk \neq 0}\cdot f_1$. By \Cref{thm:subalg-of-union-additive-functions}, $f_1'$ is union-additive. Then 
    \[f(\L) = \sum_{x\in \L} f_1'([0,x]_\L) f_2'([x,1]_\L) ~~~\text{for}~~~ \L\in \Latt_{\leq r+1}.\] 
    Finally, \Cref{prop:sum-over-simple-products} with $f_1 = f_1'$ and $f_2 = f_2'$ shows that $\sum_{x\in \L} f_1'([0,x]_\L) f_2'([x,1]_\L)$ is union-additive.
\end{proof}

\begin{proposition}(Recursion from below)\label{prop:recursion-from-below}
    Let $R$ be a commutative ring, $r\in \mathbb N_0$, $f_1,f_2,f:\Latt\to R$, $f_1$ union-additive on $\Latt_{\leq r}$, $f_2$ union-additive, such that for every $\L\in \Latt$
    $$
        f(\L) = \sum_{x\in \L, x\neq 1} f_1([0,x]_\L)f_2([x,1]_\L).
    $$
    Then $f$ is union-additive on $\Latt_{\leq r+1}$.
\end{proposition}
\begin{proof}
    Analogous to the proof of Proposition \ref{prop:recursion-from-above}.
\end{proof}

\begin{proposition}\label{prop:sum-over-k-flags}
    Let $R$ be a commutative ring and let $g:\Latt \to R$ be a union additive function. Further, let $\mathcal F_{k,\L}$ be the set of all flags $0=x_0 < x_1 < \dots < x_{k+1} =1$ in a lattice $\L$. Then 
    $$g_k(\L)=\sum_{(x_0,\dots ,x_{k+1})\in \mathcal F_{k,\L}}\prod_{i=0}^k g([x_i,x_{i+1}]_\L)$$ 
    is union-additive.
\end{proposition}
\begin{proof}
    We proceed by induction on $k$. The case $k=0$ is clear and the case $k=1$ is Proposition \ref{prop:sum-over-simple-products} with $f_1=f_2 = g$. 
    
    For the induction step, write
    $$\sum_{(x_0,\dots ,x_{k+1})\in \mathcal F_{k,\L}}\prod_{i=0}^k g([x_i,x_{i+1}]_\L) = \sum_{x\in \L}\left( g([x,1]_\L)\cdot \sum_{(x_0,\dots ,x_k)\in \mathcal F_{k-1,[0,x]_{\L}}} \prod_{i=0}^{k-1} g([x_i,x_{i+1}]_\L)\right).$$ 
    Note that for $i\leq k-1$, $g([x_i,x_{i+1}]_\L) = g([x_i,x_{i+1}]_{[0,x]_\L})$. By the induction hypothesis, we can use Proposition \ref{prop:sum-over-simple-products} with $f_1 = g_{k-1}$ and $f_2 =g$.
\end{proof}

\begin{corollary}\label{cor:sum-over-all-flags}
    Let $R$ be a commutative ring and let $g:\Latt \to R$ be a union additive function. Further, let $\mathcal F_\L$ be the set of all flags $0=x_0 < x_1 < \dots < x_{k+1} =1$ for all $k$ in a lattice $\L$. Then 
    $$g_\infty (\L)=\sum_{(x_0,\dots ,x_{k+1})\in \mathcal F_{\L}}\prod_{i=0}^k g([x_i,x_{i+1}]_\L)$$ 
    is union-additive.
\end{corollary}
\begin{proof}
    We have $g_\infty(\L) = \sum_{k=0}^{\rk(\L)} g_k(\L)$, where $g_k$ is defined as in Proposition \ref{prop:sum-over-k-flags}. Hence, we expressed $g_\infty(\L)$ as a sum of union-additive functions.
\end{proof}

\subsection{Known union-additive invariants}

With \Cref{thm:subalg-of-union-additive-functions}, it is now rather easy to see that a lot of matroid invariants are union-additive.

We define the characteristic polynomial of a lattice $\L$ by
\[ \chi(\mathcal{L})(t) := \sum_{A\in \mathcal{L}} \mu([0,A]_\L) t^{rk(A)}, \]
where $\mu:\Latt\to \mathbb{Z}$ is the M\"obius function recursively defined as $1$ on the one element lattice and
\[ \mu(\mathcal{L}) = -\sum_{0 \leq B<1} \mu([0,B]_\L).\]
We define the reduced characteristic polynomial of $\mathcal{L}$ as
\[ \overline{\chi}(\L)(t) := \frac{\chi(\L)(t)}{t-1}. \]
Note that if $\mathcal{L}$ is a lattice of flats of a matroid, which is the case in the proof below, then the definitions above coincide with \Cref{def-char}.\\

\begin{corollary}\label{cor:Mobius-function-is-union-additive}
    The Möbius function is union-additive.
\end{corollary}
\begin{proof}
    We use induction on rank. For the base of induction, every function is union-additive on $\Latt_{\leq 0}$. For the induction step, use \Cref{prop:recursion-from-below} with $f_1 = \mu$ and $f_2$ being constant $-1$.
\end{proof}

\begin{corollary}
    The characteristic polynomial, and hence also the reduced characteristic polynomial, is union-additive.
\end{corollary}
\begin{proof}
    If $t$ is the variable, the function $\mu([0,x])t^{\rk(x)} = \mu([0,x])t^{\rk([0,x])}$ is union-additive by Corollary \ref{cor:Mobius-function-is-union-additive}. Since the characteristic polynomial is defined as $\chi(\L)(t) = \sum_{x\in \L} \mu([0,x]_\L)t^{\rk(x)}$, it is union-additive by \Cref{prop:sum-over-simple-products}.

    Since the reduced characteristic polynomial is defined as $\overline{\chi}(\L)(t) = \frac{\chi(\L)(t)}{t-1}$, its union-additivity follows immediately.
\end{proof}

\begin{remark}
    There are many other union-additive functions. We are not aware of a valuative invariant of a matroid depending only on its simplification which is not union-additive. Here are some examples:
    \begin{enumerate}
        \item Chow polynomial (see \Cref{prop:Chow-is-union-additive}),
        \item flag $f$-vector and $h$-vector (for definition see \cite[Section 3]{NS04}),
        \item Kazhdan-Lusztig polynomials (for appropriate formula see \cite[Proposition 2.12]{EPW16}),
        \item Z-polynomial (for definition see \cite[Theorem 9.1]{FS22}),
        \item volume polynomial (for appropriate formula see \cite[Theorem 3.2 and Corollary 3.3]{E20}).
    \end{enumerate}
    On the other hand, the Tutte polynomial is not union-additive, as can be seen from the Tutte polynomial for $D_2$, $D_{2,1}$ and $B_2$.
\end{remark}

\begin{conjecture}
    Valuative invariants of matroids are union additive if and only if they depend only on the simplification of a matroid.
\end{conjecture}

\section{Gamma positivity of restrictions of reflection arrangements}\label{sec:gamma_pos_part}

\subsection{Simplicial arrangements and simplicial complexes}\label{subsec: prelims_gamma_pos_part}

In this subsection, we recall some results concerning simplicial complexes associated to simplicial arrangements. In particular, we review different ways of computing the $h$-polynomial of the associated complex and discuss some known $\gamma$-positivity statements.

Given a simplicial hyperplane arrangement $\A$ in $\mathbb{R}^n$ with collection of chambers $\C(\A)$. A \textit{cone} of $\A$ is an intersection of a subset of chambers from $\C(\A)$. The \textit{dimension} of a cone corresponds to the dimension of the subspace it spans.

Let $R$ be the set of rays of $\A$ (i.e. the 1-dimensional cones) and define
\[
    \Delta_\A:=\{\sigma\subseteq R:\sigma\text{ is the set of rays in some cone of }\A\}
\]
which is a simplicial complex. Alternatively, we can intersect all cones of $\A$ with the sphere $S^{d-1}$ to obtain the faces of a topological simplicial complex which is isomorphic to $\Delta_\A$. This shows that $\Delta_\A$ triangulates a $(d-1)$-sphere and hence its $h$-polynomial has non-negative and symmetric coefficients. Thus, it makes sense to speak of the $\gamma$-vector of a simplicial arrangement.

Note that in this special case, positivity and symmetry are relatively easy to prove directly. Indeed, the positivity of the $h$-vector follows from the following formula, which is a reformulation of \cite[Theorem 2.3.]{B94}:

\begin{lem}
    \label{lem: h_poly_from_sep_chambers}
    Let $\A$ be a simplicial hyperplane arrangement with chambers $\mathcal{C}$ and choose some $C_0\in\mathcal{C}$. Then 
    \[
        h(\A)(t)=\sum_{C\in\mathcal{C}}t^{\text{sep}(C)}
    \]
    where $\text{sep}(C)$ is the number of walls of $C$ which correspond to hyperplanes separating $C$ from $C_0$.
\end{lem}

This also easily implies symmetry of the $h$-vector, since we can apply the above formula with either $C_0$ or its negative $-C_0$ as the base region, noting that the walls which separate a region $C$ from $C_0$ are exactly the walls which do not separate $C$ from $-C_0$.
Altogether, we see that simplicial arrangements have a well-defined $\gamma$-vector.

We can compute the $h$-polynomial using a special directed graph attached to a reflection arrangement.

\begin{df}
    Let $\mathcal{A}$ be a simplicial arrangement and $\mathcal{C}$ its collection of chambers.
    The \textbf{tope graph} of $\mathcal{A}$ is the simple undirected graph $T=(V,E)$, where $V=\mathcal{C}$ is the set of chambers of $\mathcal{A}$ and $\{C_1,C_2\} \in E$ whenever the corresponding chambers $C_1,C_2$ share a common wall.
\end{df}

Let us fix a fundamental chamber $C_0\in\mathcal{C}$ and orient all hyperplanes in $\mathcal{A}$ such that $C_0$ is in the positive halfspace of every hyperplane. So we can represent it via a sign-vector of the form $(+,\ldots,+)$. Similarly, we can represent all the other chambers by sign vectors. Therefore, the above graph $T$ can be interpreted in the following way: Vertices correspond to sign vectors and two vertices are joined by an edge if there is a sign flip in the position associated to the shared hyperplane between them. With this interpretation, we can direct the tope graph as follows.

\begin{df}
  Let $\mathcal{A}$ be an simplicial arrangement and $T=(V,E)$ its topegraph.
  We define the \textbf{directed topegraph}, again denote by $T=(V,E)$, by directing an egde towards the vertex having the greater amount of negative signs in its sign-vector analogue.   
\end{df}

Utilizing this notion of a directed topegraph we obtain the following result, which is a direct consequence of \cite[Proposition 2.1.]{PRW08}.

\begin{lem}
\label{lem: h_poly_from_tope_graph}
    Let $\mathcal{A}$ be an simplicial arrangement and $T=(V,E)$ its directed topegraph. Then we get
    \[
      h(\mathcal{A})(t)=\sum_{v\in V}t^{\emph{\indeg}(v)},
    \]
    where $\emph{\indeg}(v)$ denotes the number of incoming edges to $v$, i.e., the in-degree of $v$.
\end{lem}

Then the fundamental question is which simplicial arrangements are $\gamma$-positive?
A first promising result is this:

\begin{thm}\cite[Theorem 1.2.]{S07}
\label{thm: all_ref_arr_gam_pos}
    All reflection arrangements are $\gamma$-positive.
\end{thm}

However, all known proofs of this fact use the classification of finite reflection groups and work on a case-by-case basis.
Another interesting result which will be important later is also the following:

\begin{thm}[\cite{G05}]
\label{thm: all_simpl_rank5_arr_gam_pos}
    All simplicial arrangements of rank $\leq 5$ are $\gamma$-positive. 
\end{thm}

For arrangements of type $A,B,D$ there is a combinatorial interpretation of the entries of the $\gamma$-vector. First, for the arrangement of type $A$, the coefficient $\gamma_k(\A_n)$ is the number of permutations $\sigma\in S_n$ with exactly $k$ descents but no double descents and no final descent, see \cite[Section 11]{PRW08} for a proof.

The types $B$ and $D$ were treated in \cite{S07}. First, we need the notion of a \textbf{peak}.
For any permutation $u\in S_n$, we denote by $p(u)$ the number of peaks, i.e., the number of positions $i\in\{1,\dots,n\}$ such that $u_{i-1}<u_i$ and $u_i>u_{i+1}$, with the convention that $u_0=0$ and $u_{n+1}=n+1$. The following result is the assembly of \cite[Corollary A.2.]{S07} and \cite[Corollary A.5.]{S07}. 

\begin{thm}
\label{thm:explicitBD}
    We have
    \begin{align*}
        h(\B_n)(t)&=\sum_{u\in S_n}(4t)^{p(u)}(1+t)^{n-2p(u)} \text{ and}\\ 
        h(\D_n)(t)&=\frac{1}{2}\sum_{u\in S_n}\phi(u)(4t)^{p(u)}(1+t)^{n-2p(u)},
    \end{align*}
    where
    \[
        \phi(u)=\begin{cases}
        2&\text{ if }u_1<u_2<u_3\\
        0&\text{ if }u_2<u_1<u_3\\
        1&\text{ otherwise}\\
        \end{cases}
    \]
\end{thm}

For type $B$, this implies
\[
    \gamma_k=4^k\cdot |\{u\in S_n:p(u)=k\}|
\]
giving the aforementioned combinatorial interpretation. The type $D$ is similar, though slightly more cumbersome because of the function $\phi$.

\subsection{Arithmeticity of the $\gamma$-vectors for $\D_{n,s}$-arrangements}\label{subsec:arthmeticity_gamma}

We apply results of Section \ref{sec: arithmeticity-of-matroid-invariants} to the $\gamma$-vector. We also determine an explicit expression for this increment in terms of the number of permutations with a certain number of maxima.

\begin{theorem}
    \label{gammaArithemtic}
    For each $n$ and $k$, the values $\gamma_k(\D_{n,s})$ form an arithmetic sequence with respect to $s$, meaning that the increment $\gamma_k(\D_{n,s+1})-\gamma_k(\D_{n,s})$ is independent of $s$.
\end{theorem}
\begin{proof}
    It is sufficient to prove that $f_k(\D_{n,i})$ increases arithemtically with $i$, since for fixed $n$, the transformations to the $h$-vector and to the $\gamma$-vector are both linear.
    
    Let $\A$ be any hyperplane arrangement. By Zaslavsky's formula (\cite{Z75}, Theorem A), number of faces of dimension $k$ contained in a flat $x\in \L(\A)$ of rank $k$ is
    $$
        \sum_{y\geq x} (-1)^{\rk([x,y])}\mu([x,y]),
    $$
    which is a union additive function by \Cref{prop:sum-over-simple-products} applied to the lattice $[x,1]_{\L(A)}$. Then
    \begin{align*}
        f_k(\L(\A)) &=\sum_{x\in \L(\A),\,\rk(x) = k} \sum_{y\geq x} (-1)^{\rk([x,y])}\mu([x,y]) \\
        &=\sum_{x\in \L(\A)}\left( \delta_{\rk=k}\cdot \sum_{y\geq x} (-1)^{\rk([x,y])}\mu([x,y])\right)
    \end{align*}   
    is union additive by \Cref{prop:sum-over-simple-products}, too. Proposition \ref{prop:arithmeticity-for-Dns} finishes the proof.
\end{proof}

\begin{corollary}
\label{coro: Dns_gam_pos}
    The arrangements $\D_{n,s}$ are $\gamma$-positive.
\end{corollary}
\begin{proof}
    By Theorem \ref{gammaArithemtic}, we have 
    \[\gamma_k(\D_{n,i})\geq \min \{\gamma_k(\D_{n,n}),\gamma_k(D_{n})\}\] 
    for all $i=0,\dots,n$. In fact, we know from the explicit formula in Theorem \ref{thm:explicitBD} that $\gamma_k(\D_{n})\leq \gamma_k(\B_{n})$, so we have $\gamma_k(\D_{n,i})\geq\gamma_k(\D_{n})$.
    Since we also know from Theorem \ref{thm:explicitBD} that the $\gamma$-vector for $D_n$ is positive, the claim follows.
\end{proof}

In the following, we express the increment in  terms of the number of so-called maxima of a permutation.

\begin{definition}
  Let $u = (u_1,\dots, u_n) \in S_n$ be a permutation in one-line notation and set $u_0=u_{n+1} = 0$. Then $m(u)$ is defined as the number of \textbf{maxima}, i.e. the number of $i \in \lbrace 1,\dots ,n\rbrace$ such that $u_{i-1} < u_i > u_{i+1}$.
\end{definition}

The notion of a maximum differs from that of a peak only by the conditions on $u_0$ and $u_{n+1}$.

\begin{definition}
    Let $u = (u_1,\dots, u_n) \in S_n$ be a permutation and set $u_0=0$, $u_{n+1} =n+1$. Then $p(u)$ is the number of \textbf{peaks}, i.e. the number of $i \in \lbrace 1,\dots ,n\rbrace$ such that $u_{i-1} < u_i > u_{i+1}$.
\end{definition}

We find the increment of the $\gamma$-vector by computing $h(\B_n,t) - h(\mathcal D_{n,n-1}, t)$. We prove the following result after proving a few auxiliary lemmata.
\begin{proposition}
\label{prop:explicit-increment}
The following holds
  \[
    h(\B_n)(t) - h(\mathcal D_{n,n-1})(t) = \frac 1 2 \sum_{u'\in S_{n-1}}(4t)^{m(u')}(1+t)^{n-2m(u')}.
  \]
\end{proposition}

As a direct consequence of \Cref{prop:explicit-increment}, we obtain the explicit increments of the entries of the $\gamma$-vector.

\begin{corollary}
    For any $n\in \mathbb N$, $i\in [i]$ and $k \in \lbrace 0,\dots, \lfloor \frac n 2\rfloor \rbrace$
    \[
        \gamma_k(\D_{n,i}) - \gamma_k(\D_{n,i-1}) = \frac 1 2\cdot |\lbrace u\in S_{n-1}: m(u) = k\rbrace|\cdot 4^{k}.
    \]    
\end{corollary}

\begin{definition}
  For $u\in S_n$ we define the \textbf{horizontal flip}, denoted by $\hf(u)$, as the permutation $(n+1-u_1, \dots, n+1 - u_n)$.
\end{definition}

\begin{definition}
  Let $u = (u_1,\dots, u_n) \in S_n$ be a permutation. We call $i\in \lbrace 2,3,\dots , n-1 \rbrace$ a \textbf{valley} if $u_{i-1} > u_i < u_{i+1}$.
\end{definition}

\begin{lemma}
  Consider $u\in S_n$ such that $u_1 = n$, and denote $u'=(u_2,\dots ,u_{n})\in S_{n-1}$. Then $p(u) = m (\hf(u'))$.
\end{lemma}
\begin{proof}
  Between every two maxima of $u'$ is a valley which turns into a peak after a horizontal flip. Since for peaks we let $u_{n+1} = n+1$ and $u_1 = n$ by assumption, $u_1$ and $u_{n+1}$ are both greater than $u_i$ for all $2\leq i\leq n$. Hence, there are no other peaks in $u$ on positions $2,\dots,n$.
  
  For any permutation, every valley is between two maxima. So, there is one less valley than maxima in $u'$. Since there is always a peak in $u$ at the position $1$, the claim follows.
\end{proof}

Let $u\in S_n$ be a permutation and $\sigma\in\lbrace 1,-1\rbrace ^n $ be a sign vector. Then we call $w = (\sigma_1u_1,\dots, \sigma_nu_n)$ a \textbf{signed permutation}. See the appendix of \cite{S07} for the correspondence between signed permutations and element of the Weyl group of type $B$ and $D$, which in turn correspond to regions in the hyperplane arrangements.
For a signed permutation $w$, we denote by $R_{w}^B$ the associated region in $\B_n$, and if the number of $-1$ in $\sigma$ is even, we denote by $R_w^D$ the corresponding region in $\D_n$, that is the region which contains the point $(\sigma_{v_1}v_1,\dots, \sigma_{v_n}v_n)$, where $v = u^{-1}$ as permutations.

In the following, let us fix the base region $R_0$ corresponding to $(1,2,\dots , n)$. We analyze the contribution of regions to the $h$-vector with respect to $R_0$.

\begin{definition}
  For a region $R$ in a hyperplane arrangement $\A$ we denote by $h(\A,R)$ the power of $t$ with which $R$ contributes to the $h$ polynomial with respect to the base region $R_0$.
\end{definition}

By \Cref{lem: h_poly_from_sep_chambers}, $h(\A,R)$ is the number of walls of $R$ in $\A$ which separate $R$ from the point $(1,2,\dots,n)$.

\begin{definition}\label{def: descents-B-and-D}
  For a signed permutation $w$, we denote by $D_B(w)$ the set of $i\in [n]$ such that $w_{i-1} > w_i$, where we set $w_0 = 0$. If the number of minus signs in $w$ is even, we denote by $D_D(w)$ a subset of $[n]$, such that $1\in D_D(w)$ if $-w_1 > w_2$ and for $i>1$, $i\in D_D(w)$ if $w_{i-1} > w_i$.
\end{definition}

Then $|D_B(w)| = h(B,R_w)$ and $|D_D(w)| = h(D,R_w)$, see appendix in \cite{S07}.

The hyperplane $\lbrace x_n = 0 \rbrace$ contains wall of exactly those regions whose canonical point $(\sigma_{v_1}v_1, \dots, \sigma_nv_n)$ satisfies $v_n = 1$, that is $u_1 = (v^{-1})_1 = n$. Hence, any region $R_{\sigma u}^B$ such that $u_1\neq n$ is in $\D_{n,n-1}$ as well, since the hyperplane $\lbrace x_n = 0 \rbrace$ does not contain any of its walls. Then $h(\D_{n,n-1},R_{\sigma u}^B) = h(B_{n}, R_{\sigma u}^B)$, since this is a question of how many walls of $R_{\sigma u}^B$ separate it from the point $(1,2,\dots,n)$.

Let $u\in S_n$ be such that $u_1 = n$, let $\sigma,\tau\in \lbrace 1,-1\rbrace^n$ be such that they differ only in the first coordinate and $\sigma_1 = 1$, $\tau_1 = -1$. Then there is the region $R:=R_{\sigma u}\cup R_{\tau u}$ in $\D_{n,n-1}$, since $R_{\sigma u}$ and $R_{\tau u}$ share a wall in the hyperplane $\lbrace x_n = 0\rbrace$. 

\begin{lemma}
  In the setting above, $h(\B_n,R_{\sigma u}^B) = h(\B_n, R_{\tau u}^B) = h(\D_{n,n-1},R)$.
\end{lemma}
\begin{proof}
  Since the interior of every region in $\D_n$ is intersected exactly by one coordinate hyperplane and $R$ is intersected by $\lbrace x_n=0\rbrace$, the region $R$ is also in $\D_n$. It corresponds in $\D_n$ to one of the signed permutations $\sigma u$, $\tau u$, let us denote this permutation $w$. We only need to compare the following three sets: $D_B(\sigma u)$, $D_B(\tau u)$ and $D_D(w)$. Using \Cref{def: descents-B-and-D}, we see that those sets can differ only in the containment of 1 and 2.
  
  Since $u_1 = n$, $1\not\in D_B(\sigma u)$, $2\in D_B(\sigma u)$, $1\in D_B(\tau u)$, $2\not\in D_\B(\tau u)$. If $R$ corresponds to $\sigma u$, $D_D(\sigma u)$ contains $2$ and not $1$. Otherwise, $D_D(\tau u)$ contains $1$ and not $2$. Each of the sets contains exactly one of 1 and 2, so their sizes are equal, as desired.
\end{proof}

\begin{proof}[Proof of Proposition \ref{prop:explicit-increment}]
  In the view of the two lemmas above and \Cref{thm:explicitBD}
  \begin{align*}
       h(\D_{n,n-1})(t) = &\sum_{R\in \D_{n,n-1}} t^{h(\D_{n,n-1},R)}=\\
    =& \sum_{u\in S_n, u_1\neq n} (4t)^{p(u)}(1+t)^{n-2p(u)} + \frac 1 2 \sum_{u\in S_n, u_1=n} (4t)^{p(u)}(1+t)^{n-2p(u)}, 
  \end{align*}
  and
  \begin{align*}
        \sum_{u\in S_n, u_1=n} (4t)^{p(u)}(1+t)^{n-2p(u)} &= \sum_{u'\in S_{n-1}} (4t)^{m(\hf(u'))}(1+t)^{n-2m(\hf(u'))}\\
        &= \sum_{u'\in S_{n-1}} (4t)^{m(u')}(1+t)^{n-2m(u')}.
   \end{align*}
   
  Moreover
  \[
    h(\B_n)(t) = \sum_{u\in S_n, u_1\neq n} (4t)^{p(u)}(1+t)^{n-2p(u)} + \sum_{u\in S_n, u_1=n} (4t)^{p(u)}(1+t)^{n-2p(u)}, 
  \]
  hence we obtain the desired result.
\end{proof}

We end this subsection with some illustrative examples of $\gamma$-vectors in small dimensions.

\begin{ex}
    Let $n\in\{3,\dots,6\}$ and $s\in\{0,\dots,6\}$. By using \verb|SageMath| we compute the $\gamma$-vectors of the arrangements $\D_{n,s}$, shown in below \Cref{table: gam_vecs_Dns_n36_s06}. The code used for these computations is available in the GitHub repository \cite{DHMPR25code}.
\end{ex}

\begin{center}
\begin{table}[h]
    \ra{1.3}
    \begin{tabular}{@{}lllllllll@{}} 
     \hline
     $n$ & &$s=0$ & $s=1$ & $s=2$ & $s=3$ & $s=4$ & $s=5$ & $s=6$ \\
     \hline
     
     $3$& $\begin{matrix}
         \gamma_0\\
         \gamma_1
     \end{matrix}$&
     $\begin{matrix}
             1\\
             8
           \end{matrix}$& $\begin{matrix}
             1\\
             12
           \end{matrix}$& $\begin{matrix}
             1\\
             16
           \end{matrix}$& $\begin{matrix}
             1\\
             20
           \end{matrix}$ & & &\\[0,2cm]
           
    \hline 
     
     $4$ & $\begin{matrix}
         \gamma_0\\
         \gamma_1\\
         \gamma_2
     \end{matrix}$&$\begin{matrix}
             1\\
             40\\
             16
           \end{matrix}$ & $\begin{matrix}
             1\\
             48\\
             32
           \end{matrix}$ & $\begin{matrix}
             1\\
             56\\
             48
           \end{matrix}$ & $\begin{matrix}
             1\\
             64\\
             64
           \end{matrix}$ & $\begin{matrix}
             1\\
             72\\
             80
           \end{matrix}$ & &\\
           
    \hline 
       
     $5$ & $\begin{matrix}
         \gamma_0\\
         \gamma_1\\
         \gamma_2
     \end{matrix}$&$\begin{matrix}
             1\\
             152\\
             366
           \end{matrix}$ & $\begin{matrix}
             1\\
             168\\
             464
           \end{matrix}$ & $\begin{matrix}
             1\\
             184\\
             592
           \end{matrix}$ & $\begin{matrix}
             1\\
             200\\
             720
           \end{matrix}$ & $\begin{matrix}
             1\\
             216\\
             848
           \end{matrix}$ & $\begin{matrix}
             1\\
             232\\
             976
           \end{matrix}$ &\\
        
      \hline
           
     $6$ & $\begin{matrix}
         \gamma_0\\
         \gamma_1\\
         \gamma_2\\
         \gamma_3
     \end{matrix}$&$\begin{matrix}
             1\\
             524\\
             3440\\
             832
           \end{matrix}$ & $\begin{matrix}
             1\\
             556\\
             4144\\
             1344
           \end{matrix}$ & $\begin{matrix}
             1\\
             588\\
             4848\\
             1856
           \end{matrix}$ & $\begin{matrix}
             1\\
             620\\
             5552\\
             2368
           \end{matrix}$ & $\begin{matrix}
             1\\
             652\\
             6256\\
             2880
           \end{matrix}$ & $\begin{matrix}
             1\\
             684\\
             6960\\
             3392
           \end{matrix}$ & $\begin{matrix}
             1\\
             716\\
             7664\\
             3904
           \end{matrix}$\\ 
     \hline
    \end{tabular}
    \caption{The $\gamma$-vectors of $D_{n,s}$ for $n=3,\dots,6$ and $s=0,\dots,6$.}
    \label{table: gam_vecs_Dns_n36_s06}
\end{table}
\end{center}

\subsection{The exceptional cases}\label{subsec:exceptional}

In this subsection we discuss the restrictions of rank $\geq 6$ of the exceptional reflections types $E_7$ and $E_8$. In particular, we computationally verify the $\gamma$-positivity of those restrictions and then finally prove our first main theorem (\Cref{thm: intro_gam_pos}).

For the exceptional reflection types $E_7$ and $E_8$, there exist four restrictions of rank $\geq 6$. More precisely, $E_7$ has one restriction of rank $6$, and $E_8$ has two restrictions of rank $6$ and one of rank $7$. Following \cite[Figure 9]{MS22} we denote these four arrangements by $(46,1)$, $(63,1)$, $(68,1)$ and $(91,1)$, respectively. Here the first entry in each pair encodes the number of hyperplanes of the corresponding restriction.

\begin{proposition}
\label{prop: gam_pos_exceptional}
    The $\gamma$-vectors of the restrictions $(46,1)$, $(63,1)$, $(68,1)$ and $(91,1)$ are given by the vectors in \Cref{table: gam_vecs}. In particular, these arrangements are $\gamma$-positive.  
\end{proposition}

\begin{proof}
    This proof is purely based on computations using the computer algebra system \verb|SageMath|. For the first three simplicial arrangements we start by computing their undirected topegraphs using the algorithm from \cite[Appendix A]{KSSW25} and direct them as described in \Cref{subsec: prelims_gamma_pos_part}. Now we are able to apply \Cref{lem: h_poly_from_tope_graph} to compute the $h$-polynomials associated to the three arrangements. Lastly we expand these $h$-polynomials in the $\gamma$-basis and obtain the first three $\gamma$-coefficient vectors displayed in \Cref{table: gam_vecs}. For the simplicial arrangement $(91,1)$ the aforementioned procedure fails due to the computational complexity of its topegraph. Michael Cuntz computed the $f$-vector of this restriction by using the structure of the associated Weyl groupoid and exploiting its symmetries, see \cite{weylgroupoids} for more details. Thus we obtain the $\gamma$-coefficient vector in \Cref{table: gam_vecs} by transforming the $f$-vector into the $h$-vector and again using the $\gamma$-expansion.     
\end{proof}

\begin{center}
  \begin{table}[h]
  \centering
  \ra{1.3}
  \begin{tabular}{@{}lllll@{}}
   \hline
    $\A$ & $(46,1)$ & $(63,1)$&$(68,1)$&$(91,1)$\\
    \hline
    $\gamma_0$ & $1$ & $1$ & $1$ & $1$ \\
    $\gamma_1$ & $3468$ & $14064$ & $8604$ & $119940$\\
    $\gamma_2$ & $24048$ & $112080$ & $152560$ & $1822512$\\
    $\gamma_3$ & $9536$ & $52352$ & $59712$ & $2403008$\\
       \hline
    \end{tabular}
    \caption{$\gamma$-vectors of the restrictions $(46,1)$, $(63,1)$, $(68,1)$ and $(91,1)$.}
    \label{table: gam_vecs}
    \end{table}
  \end{center} 
\vspace{-20pt}

\begin{remark}
    The \verb|SageMath|-implementations of directing the tope graphs, computing their $h$-polynomials via \Cref{lem: h_poly_from_tope_graph} and expanding these in the $\gamma$-basis are available in the GitHub repository \cite{DHMPR25code}. Furthermore the tope graphs of the first three restrictions in \Cref{prop: gam_pos_exceptional} can be found as database files in the GitHub repository \cite{KSSW25code}.
\end{remark}

Now with all preparations done, we can finally prove the first main result of the article.

\begin{theorem}[\Cref{thm: intro_gam_pos}]
    All restrictions of reflection arrangements are $\gamma$-positive.
\end{theorem}

\begin{proof}
    By \Cref{thm: all_simpl_rank5_arr_gam_pos} all simplicial arrangements of rank $\leq 5$ are $\gamma$-positive, so we only have to consider those restrictions of reflection arrangements that are of rank $\geq 6$. According to \cite[Figure 9]{MS22} the only restrictions of reflection arrangements of rank $\geq 6$ are the $\D_{n,s}$ for $n\geq 6$ and the four exceptional restrictions of $E_7$ and $E_8$, namely $(46,1)$, $(63,1)$, $(68,1)$ and $(91,1)$. Now the $\gamma$-positivity for the $\D_{n,s}$ cases is covered by \Cref{coro: Dns_gam_pos} and for the restrictions $(46,1)$, $(63,1)$, $(68,1)$ and $(91,1)$ by \Cref{prop: gam_pos_exceptional}, which concludes the proof.
\end{proof}

\section{Chow polynomials of reflections arrangements}\label{sec:chow_polyn_part}

\subsection{Chow polynomials of matroids}\label{subsec:prelims_chow_polyn}
In this subsection we briefly review a general combinatorial description of the Chow polynomial associated to a simple matroid, introduced by Stump in \cite{S24}, which was already mentioned in the introduction.

In order to state this combinatorial description we first define a special labeling of a finite graded poset with a unique top and bottom element.

Let $(P,\preceq)$ be a bounded poset with the top element $\hat{1}$ and the bottom element $\hat{0}$. For $x,y\in P$ we denote by $x\prec y$ the strict relation and by $x \precdot y$ the cover relation. Furthermore we set $\text{Cov}(P) \subset P \times P$ to be the set of cover relations in $P$.

\begin{df}
    \label{def: edge_labeling}
    An \textbf{edge-labeling} of $P$ is a map $\lambda \colon \text{Cov}(P) \rightarrow \Lambda$ where $\Lambda$ is another poset. Given such a labeling, each saturated chain $c=(x \precdot z_1 \precdot \dots \precdot z_r \precdot y)$ between two elements $x \preceq y $ is called \textbf{increasing} if $\lambda(c) = (\lambda(x,z_1),\dots,\lambda(z_r,y))$ is increasing, and \textbf{decreasing} if $\lambda(c)$ is strictly decreasing. Finally, we call $\lambda$ an \textbf{$R$-labeling} if for all $x\prec y$ in $P$ there exist a unique increasing saturated chain between them. 
\end{df}

Now let $\mathcal{L}(M)$ be the lattice of flats of a matroid $M$ with integer-valued $R$-labeling $\lambda$. To a maximal chain $\mathcal{F} = \{F_0 \prec \dots \prec F_n\}$ we associate the sequence $\lambda(\mathcal{F}) = (\lambda_1,\dots,\lambda_n)$ of edge labels $\lambda_i = \lambda(F_{i-1} \prec F_i)$, where $n$ denotes the rank of $M$.

The following theorem is an important tool for the subsequent discussion. 

\begin{thm}\cite[Theorem 1.1.]{S24}
\label{thm:chow_formula}
The Chow polynomial $H(M)(t)$ can be computed as 
\[ H(M)(t) = \sum_\mathcal{F} t^{\text{des}(\lambda_\mathcal{F})} (t+1)^{n-1-2\text{des}(\lambda_\mathcal{F})}, \]
where the sum ranges over all maximal saturated chains $\mathcal{F}$ in $\mathcal{L}(M)$ such that for $\lambda_{\mathcal{F}} = (\lambda_1,\dots,\lambda_n)$ we have that $\lambda_1 < \lambda_2$ and that $\lambda_i > \lambda_{i+1}$ implies $\lambda_{i-1} < \lambda_i$ for $i \in \{2,\dots n-1\}$.
\end{thm}

For matroids arising from arrangements of type $\A_{n-1}$, i.e., \textit{braid matroids}, there exists an explicit formula using permutation statistics.

\begin{theorem}\cite[Theorem 3.1.]{S24}
\label{thm:chow_formula_A_type}
    The Chow polynomial of the braid matroid is
    \[
        H(\A_{n-1})(t)=\sum_{(a_1,\ldots,a_n)} a_1\cdots a_nt^{\text{des}(a_1,\ldots,a_n)}(t+1)^{n-1-2\text{des}(a_1,\ldots,a_n)},
    \]
    where the sum ranges over all tuples $(a_1,\ldots,a_n)$ with $a_i\in\{1,\ldots,n+1-i\}$ such that $a_1 \leq a_2$ and $a_i > a_{i+1}$ implies $a_{i-1}\leq a_i$ for $i\in\{2,\ldots,n-1\}$.
\end{theorem}

We end with some examples of Chow polynomials corresponding to reflection arrangements of type $A$ and $B$. They are computed using \Cref{thm:chow_formula_A_type} and \Cref{thm:chow_formula} for type $A$ and $B$ respectively.

\begin{ex}
   For small $n$, we get the following polynomials:
   \begin{center}
       \begin{table}[h]
           \centering
           \ra{1.3}
           \begin{tabular}{@{}ll@{}}
              \hline
              \textbf{Type $A$}  & \textbf{Type $B$}  \\
              \hline
              $H(\A_2)(t)=t+1$  & $H(\B_2)(t)=t+1$  \\
              $H(\A_3)(t)=t^2+8t+1$  & $H(\B_3)(t)=t^2+14t+1$\\
            $H(\A_4)(t)=t^3 + 41t^2 + 41t + 1$ & $H(\B_4)(t)=t^3 + 99t^2 + 99t + 1$  
           \end{tabular}
       \end{table}
   \end{center}
\end{ex}

\subsection{The Chow polynomial in type $B$}\label{subsec:chow_polyn_Bn}

In this subsection we derive an explicit formula of the Chow polynomial for matroids arising from the arrangements of type $B$. Our discussion uses type $B$ analogues of the approach in \Cref{thm:chow_formula_A_type}, by defining an $R$-labeling for the intersection lattices of $\B_n$-arrangements and applying \Cref{thm:chow_formula}.

\begin{definition}
\label{def: inters_lat_B_n}
Define $\Pi_n^B$ as the set of all partitions $\pi$ of $\langle n\rangle :=\{-n,\dots,n\}$ with the properties
\begin{enumerate}
\item If $B\in \pi$ then $-B\in \pi$
\item If $i\in B$ and $-i\in B$ for some $i$ and some $B\in \pi$, then $0\in B$.
\end{enumerate}
This defines a poset by setting $x\prec y$ if and only if for each $B\in x$ there is $B'\in y$ with $B\subseteq B'$.
\end{definition}

For any $\pi\in\Pi_n^B$, the block containing $0$ is called the \textbf{zero block}.
We define the rank function $r:\Pi_n^B\to \mathbb{N}$ by
\[r(\pi)=n-\frac{|\pi|-1}{2}\]
which turns $\Pi_n^B$ into a graded poset.
Furthermore, we get the following identification, which is a reformulation of \cite[Theorem 7]{DNG19}:

\begin{lemma}
The intersection lattice for the hyperplane arrangement of type $\B_n$ is isomorphic to $\Pi_n^B$.
\end{lemma}

We can explicitly describe the cover relations in $\Pi_n^B$. To this end, let $\pi,\pi'\in\Pi_n^B$ such that $\pi\precdot \pi'$. There are two possible cases:
\begin{itemize}
\item There is a block $B\in\pi$ such that $B\cup B_0\cup (-B)\in \pi'$, where $B_0$ is the zero block in $\pi$, and all other blocks in $\pi'$ are exactly the blocks of $\pi$ which are not equal to $B,B_0,(-B)$.
\item There are two blocks $B, B'\in\pi$, neither of which is the zero block, such that $B\cup B'\in \pi'$ and $(-B)\cup (-B')\in \pi'$ and all other blocks in $\pi'$ are exactly the blocks of $\pi$ which are not equal to $B,B',(-B),(-B')$.
\end{itemize}

Next, we define a labeling for all cover relations in $\Pi_n^B$, motivated by the usual max-of-min labeling on the (unsigned) partition lattice. Let again $\pi,\pi'\in\Pi_n^B$ be such that $\pi\precdot\pi'$.
Then there exist two distinct blocks $B_1,B_2\in \pi$ which are not the zero block and which are both contained in the same block $B'$ in $\pi'$. We define the label of the cover relation $\pi\precdot\pi'$ as
\[\lambda(\pi,\pi')=\max\{\min |B_1|,\min |B_2|\}\]
where $|B|=\{|b|:b\in B\}$ is the set of absolute values of elements in the block $B$.
Note that this is well defined in both cases of cover relations described above, since it does not matter if we choose the pair $B_1,B_2$ or $-B_1,-B_2$.
\begin{lemma}
\label{setOfLabels}
Let $x,y\in \Pi_n^B$ with $x\prec y$. For any saturated chain $x=x_0\precdot\dots\precdot x_k=y$
the set of labels
\[\{\lambda(x_i,x_{i+1}): i=0,\dots,k-1\}\] is the same. We denote this set by $\lambda(x,y)$.
\end{lemma}
\begin{proof}
Let $x=\{B_1,\dots,B_k\}$ and $y=\{B_1',\dots,B_l'\}$ and consider a saturated chain $x=x_0\precdot\dots\precdot x_k=y$. For $j=1,\dots,k$, define $b_j=\min |B_j|$ and for $i=1,\dots,l$ define \[M_i=\{b_j: B_j\subseteq B_i'\}\] which is the set of absolute minima of all blocks in $x$ which are merged together to obtain the block $B_i'$ in $Y$. For each $\mu\in M_i\setminus\{\min M_i\}$, there must be some step in the chain where the block containing $\mu$ is merged with another block whose absolute minimum is smaller than $\mu$. At this step, $\mu$ appears as a label. Conversely, at any merge going from some $x_i$ to $x_{i+1}$, we use the absolute minimum of one of the blocks of $x_i$ as a label. But the absolute minimum of any block of $x_i$ already appears as the absolute minimum of a block of $x$, i.e. as one of the values $b_j$. Furthermore, it cannot be equal to $\min M_j$ for any $j=1,\dots,l$.
Thus, we have shown that the set of labels is \[\bigcup_{i=1}^l M_i\setminus\{\min M_i\}\]
which is independent of the chain we used.
\end{proof}

For any saturated chain $c$ with $x=x_0\precdot\dots\precdot x_k=y$, denote by 
\[\lambda(c)=(\lambda(x_0,x_1),\dots,\lambda(x_{k-1},x_k))\]
the sequence of labels along $c$.

For any permutation $\sigma$ of $[k]$, the inversion sequence $\text{Inv}(\sigma)=(a_1,\dots,a_k)$ is given by
\[a_i=|\{j\geq i:\sigma(j)\leq\sigma(i)\}|.\]

\begin{lemma}
\label{numberOfChains}
Let $x,y\in \Pi_n^B$ with $x\prec y$ and let $\lambda(x,y)=\{b_1,\dots,b_k\}$ with $b_1<\dots<b_k$. For any permutation $\sigma$ of $[k]$ with inversion sequence $\text{Inv}(\sigma)=(a_1,\dots,a_k)$, the number of saturated chains $c$ starting from $x$ and ending in $y$ with $\lambda(c)=(b_{\sigma(1)},\dots,b_{\sigma(k)})$ is equal to $\prod_{i=1}^k(2a_i-1)$.
\end{lemma}
\begin{proof}
For any $i\in \{1,\dots,k\}$, let $\sigma^{(i)}=(\sigma(1),\dots,\sigma(i))$ be the partial permutation obtained by taking the first $i$ entries of $\sigma$. 
Let $m(\sigma^{(i)})$ be the number of saturated chains $C$ starting from $X$ with $\lambda(C)=(b_{\sigma(1)},\dots,b_{\sigma(i)})$. 
We claim that
\begin{equation}
\label{eq}
m(\sigma^{(i)})=m(\sigma^{(i-1)})\cdot (2a_i-1).
\end{equation}
For any saturated chain $c$ with $\lambda(c)=(b_{\sigma(1)},\dots,b_{\sigma(i-1)})$, we want to count the number of ways in which $c$ can be extended so that the next label is $b_{\sigma(i)}$.
In other words, if $\pi\in\Pi_n^B$ is the endpoint of $c$, then we want to count the number of ways we can merge blocks of $\pi$ such that $b_{\sigma(i)}$ is the label assigned to the cover relation.
One of the blocks involved in such a merge must contain $b_{\sigma(i)}$, say $B$ (which is not the zero block).
The first case is that we merge $B$ to its negative and the zero block, which gives the label $b_{\sigma(i)}$.
The second case is that we merge $B$ to some block $B'$ which is not $-B$ or the zero block. In this case, let $b_{\sigma(j)}$ be the absolute minimum of $B'$, then we must have $j>i$ and $b_{\sigma(j)}<b_{\sigma(i)}$. The first property comes from the fact that all values $b_{\sigma(j)}$ with $j<i$ have already appeared as a label in $C$, meaning that they are no longer the absolute minima in their respective blocks. The second property comes from the fact that we want \[b_{\sigma(i)}=\max\{b_{\sigma(i)}, b_{\sigma(j)}\}\] and it is equivalent to $\sigma(j)< \sigma(i)$.
Note that every index $j$ with $j>i$ and $\sigma(j)>\sigma(i)$ corresponds to exactly two possible merges (merging $B$ with $B'$ or $B$ with $-B'$).
Thus, the number of ways to merge blocks of $\pi$ such that the label is $b_{\sigma(i)}$ is
\[2\cdot |\{j>i:\sigma(j)<\sigma(i)\}|+1=2a_i-1\]
which implies Equation \eqref{eq}. 
Applying Equation \eqref{eq} recursively finally yields $m(\sigma) = \prod_{i=1}^k(2a_i-1)$. 
\end{proof}

\begin{lemma}
The map $\lambda$ is an $R$-labeling of $\Pi_n^B$.
\end{lemma}
\begin{proof}
Let $x,y\in \Pi_n^B$ be two partitions with $x\prec y$ and let $\lambda(x,y)=\{b_1,\dots,b_k\}$ with $b_1<\dots<b_k$. First, we show that there exists a saturated increasing chain from $x$ to $y$. Let $B'$ be the block in $Y$ containing $b_1$ and let $b=\min |B'|$. Let $x'$ be the partition obtained from $x$ by merging the block containing $b_1$ with the block containing $b$ (and all other merges to ensure that $x'\in \Pi_n^B$). Then $x'$ covers $x$ and $\lambda(x,x')=b_1$. Furthermore we get $\lambda(x',y)=\{b_2,\dots,b_k\}$, so the existence of the chain follows from induction on $k$. For uniqueness, note that by Lemma \ref{setOfLabels}, any saturated chain from $x$ to $y$ must have exactly the labels $\lambda(x,y)$. By Lemma \ref{numberOfChains}, the number of saturated chains $c$ starting from $x$ and ending in $y$ which are increasing (i.e. with $\lambda(c)=(b_1,\dots,b_k)$) is equal to 1, since the inversion sequence of the identity permutation is $(1,\dots,1)$. This proves existence and uniqueness of $c$.
\end{proof}

\begin{lemma}
\label{inversionDescents}
Let $\sigma$ be a permutation of $[n]$ with inversion sequence $\text{Inv}(\sigma)=(a_1,\dots,a_n)$.
Then for all $i\in \{1,\dots,n-1\}$ it holds
\[a_i > a_{i+1}\iff \sigma(i)>\sigma(i+1)\]
and 
\[a_i \leq a_{i+1}\iff \sigma(i)<\sigma(i+1).\]
\end{lemma}
\begin{proof}
First, assume that $a_i > a_{i+1}$ and assume for a contradiction that $\sigma(i)\leq\sigma(i+1)$. Then any $j\geq i+1$ with $\sigma(j)\leq \sigma(i)$ which is counted in $a_{i}$ also satisfies $\sigma(j)\leq \sigma(i+1)$, meaning that it is counted in $a_{i+1}$ as well. Furthermore, in $a_i$ we count $j=i$ but not $j=i+1$, whereas we do count $j=i+1$ in $a_{i+1}$. This implies $a_i\leq a_{i+1}$, contradiction to our assumption.
Conversely, assume that $\sigma(i)>\sigma(i+1)$. Then any $j\geq i+1$ with $\sigma(j)\leq \sigma(i+1)$ which is counted in $a_{i+1}$ also satisfies $\sigma(j)\leq \sigma(i)$, meaning that it is counted in $a_{i}$ as well. Furthermore, in $a_i$ we count $j=i$ which is not counted in $a_{i+1}$. This implies $a_i > a_{i+1}$.
Lastly, the second equivalence follows from the first by contraposition.
\end{proof}

\begin{lemma}\cite[Proposition 1.3.12]{S97}
\label{inversionBijection}
The map $\sigma\mapsto \text{Inv}(\sigma)$ is a bijection between permutations of $[n]$ and sequences $(a_1,\dots,a_n)$ with $a_i\in \{1,\dots,n+1-i\}$.
\end{lemma}

Now with all preparations done, we can finally prove the second main result of the article. 

\begin{theorem}[\Cref{thm: intro_chow_poly_Bn}]
\label{thm: Chow_poly_Bn}
The Chow polynomial for type $B$ has the expansion
\[H(\B_n)(x)=\sum_{(a_1,\dots,a_n)}\left(\prod_{i=1}^n (2a_i-1)\right)x^{des(a_1,\dots,a_n)}(x+1)^{n-1-2des(a_1,\dots,a_n)}\]
where the sum ranges over all tuples $(a_1,\dots,a_n)$ with $a_i\in \{1,\dots,n+1-i\}$ such that $a_1\leq a_2$ and also that $a_i>a_{i+1}$ implies $a_{i-1}\leq a_{i}$ for $i\in\{2,\dots,n-1\}$.
\end{theorem}
\begin{proof}
We have
\begin{align*}
H(\B_n)(x)&=\sum_{\mathcal F}x^{des(\lambda(\mathcal F))}(x+1)^{n-1-2des(\lambda(\mathcal F))}\\
&=\sum_{\sigma\in S_n}\sum_{\substack{\mathcal F\\ \lambda(\mathcal F)=(b_{\sigma(1)},\dots,b_{\sigma(n)})}}x^{des(\sigma)}(x+1)^{n-1-2des(\sigma)}\\
&=\sum_{\sigma\in S_n}\left(\prod_{i=1}^n (2a_i-1)\right)x^{des(\sigma)}(x+1)^{n-1-2des(\sigma)}\\
&=\sum_{(a_1,\dots,a_n)}\left(\prod_{i=1}^n (2a_i-1)\right)x^{des(a_1,\dots,a_n)}(x+1)^{n-1-2des(a_1,\dots,a_n)}.
\end{align*}
The first equality is Theorem \ref{thm:chow_formula} and the sum is indexed by the maximal saturated chains $\mathcal{F}$ with no double descents and no descent in the first position. In the second step, we split the sum according to the labels of the maximal chains and observe that $b_{\sigma(i)}>b_{\sigma(i+1)}$ is equivalent to $\sigma(i)>\sigma(i+1)$, allowing us to replace the descents of $\lambda(\mathcal F)$ by the descents of $\sigma$. In the third step we use Lemma \ref{numberOfChains}. In the last step we use Lemma \ref{inversionBijection} and Lemma \ref{inversionDescents} to switch the index set of the summation from permutations to inversion sequences and to replace the descents of $\sigma$ by the descents of $(a_1,\dots,a_n)$.
\end{proof}

\subsection{Chow polynomial arithmeticity for $\D_{n,s}$-arrangements}\label{subsec:arthmeticity_chow_EL}

In this subsection we show that the Chow polynomials associated to matroids arising from $\D_{n,s}$-arrangements satisfy an arithmeticity property with respect to a fixed $n$ and varying $s$. There are two approaches, one of them  uses a recursive formula for Chow polynomials of matroids in terms of their minors and the corresponding reduced characteristic polynomials, the second utilizes the usage of EL-labelings of the their intersection lattices and is inspired by the discussions in \cite{DNG19}.

The main result of this subsection is the following:
\begin{thm}
\label{prop:chow_arithmeticity_EL}
\label{thm1}
\label{thm: chow_arithmeticity_char_polyn}
    The Chow polynomial is arithmetic on the arrangements $\D(n,s)$. That is,
    \[ H(\D_{n,n})(t) -H(\D_{n,n-1})(t) = \dots = H(\D_{n,1})(t) -H(\D_{n,0})(t). \]
    In particular
    \[ H(\D_{n,s})(t) = \frac{s}{n} H(\B_{n})(t) + \frac{n-s}{n} H(\D_{n})(t). \]
\end{thm}

In the first proof of \Cref{thm1} we use the following result of \cite{FMSV24}:

\begin{theorem}\label{thm:chow-formula-via-char-poly}
    Let $M$ be a loopless matroid. Then, the Chow polynomial of $M$ satisfies
    \begin{align*}
        H(M) = \sum_{F\in \mathcal{L}(M),\ F\neq \emptyset} \overline{\chi}(M|_{F})H(M/F),
    \end{align*}
    where $M|_{F}$ is the restriction of $M$ to $F$ and 
    $M/F$ the contraction of $M$ to $F$.
\end{theorem}

Recursively applying this to a matroid $M$ yields the following corollary.

\begin{cor}\label{cor1}
    The Chow polynomial of a loopless matroid satisfies:
    \begin{align*}
        H(M) = \sum_{\mathcal{C}} \prod_j \overline{\chi}([C_j,C_{j+1}]_{\L(M)}),
    \end{align*}
    where the sum runs over all flags of flats in $M$ starting at $\emptyset$ and ending in $M$ and $\overline{\chi}([C_j,C_{j+1}]_{\L(M)})(t)$ denotes the reduced characteristic polynomial of the lattice of flats restricted to $[C_j,C_{j+1}]$.
\end{cor}

\begin{proposition}\label{prop:Chow-is-union-additive}
    The Chow polynomial of a loopless matroid is union-additive.
\end{proposition}
\begin{proof}
    The claim follows by union-additivity of the reduced characteristic polynomial, \Cref{thm:chow-formula-via-char-poly}, Proposition \ref{prop:recursion-from-above} and induction.

    Another proof is via Proposition \ref{cor:sum-over-all-flags} and \Cref{cor1}. Yet another proof is via the union-additivity of the flag $h$-vector and \cite[Theorem 4.25]{FMV24}.
\end{proof}

\begin{proof}[Proof of \Cref{prop:chow_arithmeticity_EL} via characteristic polynomial]
    The claim follows from \Cref{prop:Chow-is-union-additive} and \Cref{prop:arithmeticity-for-Dns}.
\end{proof}

Now we turn to the approach via EL-labeling.

Recall from \Cref{def: inters_lat_B_n} that the intersection lattice of $\B_n$- arrangements can be viewed in the following way:

Define $\Pi_n^B$ as the set of all partitions $\pi$ of $\langle n\rangle :=\{-n,\dots,n\}$ with the properties
\begin{enumerate}
\item If $B\in \pi$ then $-B\in \pi$
\item If $i\in B$ and $-i\in B$ for some $i$ and some $B\in \pi$, then $0\in B$.
\end{enumerate}
This defines a poset by setting $x\prec y$ if for each $B\in x$ there is a $B'\in y$ with $B\subseteq B'$. For any $\pi\in\Pi_n^B$, the block containing $0$ is called the \textbf{zero block}. In the following we denote the zero block by a subscript $0$.

\begin{rem}
\label{rem: lattice_translation}
    Due to \cite[Theorem 7]{DNG19} the above intersection lattice can be translated to a lattice $\Pi_{n,\Sigma}$ of signed partitions of $[[n]]:=\{\bar{1},1,\ldots,\bar{n},n\}$ as introduced in \cite{DNG19}. Therefore we set $\Sigma:=\{0\}$, $-k\in\langle n\rangle$ becomes
    $\bar{k}\in[[n]]$ and a negative block $-R$ translates to $\bar{R}$. 
\end{rem}

Next we would like to apply the EL-labeling theory in \cite{DNG19} to our setting. In order to do so, we have to translate their concepts regarding the lattice $\Pi_{n,\Sigma}$ into the language of the lattice $\Pi_n^B$, using the correspondence described in \Cref{rem: lattice_translation}.

The following definitions and results are to our setting adjusted version of \cite{DNG19}.
We start with a slight modification of \Cref{def: edge_labeling} to simplify the below discussion.

\begin{df}
    Let $P$ be a bounded poset with top element $\hat{1}$ and bottom element $\hat{0}$. Let $\text{Cov}(P) \subset P \times P$ be the set of cover relations in $P$.
    An \textbf{edge-labeling} of $P$ is a map $\lambda \colon \text{Cov}(P) \rightarrow \Lambda$ where $\Lambda$ is a poset. Given such a labeling, each maximal chain $c=(x \precdot z_1 \precdot \dots \precdot z_r \precdot y)$ between two elements $x \preceq y $ is called \textbf{increasing} if $\lambda(c) = (\lambda(x,z_1),\dots,\lambda(z_r,y))$ is strictly increasing, and \textbf{decreasing} if $\lambda(c)$ is decreasing. 
\end{df}

Note that the difference between \Cref{def: edge_labeling} and this definition lies in the notion of \emph{increasing}, which now requires a strictly increasing sequence of labels, and the according changes of the notion of \emph{decreasing}.

\begin{df}
    Let $P$ be a bounded ranked poset. We say a maximal chain with labeling $(a_1,\dots,a_k)$ lexicographically precedes a maximal chain with labeling $(b_1,\dots,b_k)$ if $a_i<b_i$ in the first coordinate where the labels differ. An \textbf{edge lexicographic labeling}, or \textbf{EL-labeling} for short, of $P$ is an edge labeling such that in each closed interval $[x,y] \subseteq P$ there is a unique increasing maximal chain and this chain lexicographically precedes all other maximal chains of $[x,y]$.
\end{df}

\begin{rem}
    Every EL-labeling is an $R$-labeling.
\end{rem}

The next result will be important later on in this note, when we discuss an EL-labeling for intersection lattices of the $\D_{n,s}$ arrangements.

\begin{lem}\cite[Lemma 16]{DNG19}
    \label{lem: subposet_criterion}
    Let $P$ be a bounded and ranked poset, having EL-labeling $\lambda$. Let $Q\subseteq P$ be a ranked subposet of $P$ containing $\hat{0}$ and $\hat{1}$, with rank function given by restricting the one of $P$. Then, if for all $x\preceq y$ in $Q$ the unique increasing maximal chain in $[x,y]\subseteq P$ is also contained in $Q$, the edge labeling $\lambda$
    restricted to $Q$ is an EL-labeling of $Q$.
\end{lem}

We are now able to reformulate the EL-labeling proposed in \cite{DNG19}.

\begin{df}
    For a block $R$ of some partition $\pi\in\Pi_n^B$, the \textbf{representative} of $R$, denoted by $r(R)$, is defined as the minimum element $i \in [n]$ such that $i \in R$ or $-i \in R$. A block $R \subseteq \langle n\rangle$ is called \textbf{normalized} if $r(R)$ belongs to $R$.
\end{df}

\begin{ex}
    Consider blocks $ R = \{2,-4,5\}$ and $-R = \{-2,4,-5\}$. Then $r(R) = r(-R) = 2$ and $R$ is normalized while $-R$ is not.
\end{ex}

There are three types of cover relations, that is, edges in the Hasse diagram, we have to distinguish. 

\begin{df}
\label{def:edge_coh}
An edge $\mathcal{E}_{xy}$ of the Hasse diagram of $\Pi_n^B$ is called:
\begin{itemize}
    \item \textbf{signed} if $x_0 \subsetneq y_0$, where these denote the zero blocks of $x$ and $y$, respectively.
    \item \textbf{coherent} if $R\cup R^\prime \in y$ for some normalized non-zero blocks $R, R^\prime \in x, R\neq R^\prime$;
    \item \textbf{non-coherent} if $R \cup -{R^\prime} \in y $ for some normalized non-zero blocks $R, R^\prime \in x, R\neq R^\prime $
\end{itemize}
An edge $\mathcal{E}_{xy}$ is called \textbf{unsigned} if it is either coherent or non-coherent.
\end{df}

Using all of the above mentioned notions we are able to state a slightly modified version of the edge-labeling in \cite{DNG19}.

\begin{df}
\label{def:concrete_el_labeling}
    We define the following edge labeling $\lambda$ of $\Pi_n^B$.
    \begin{itemize}
        \item Let $\mathcal{E}_{xy}$ be an unsigned edge and let $R, R^\prime \in x $ such that $R \cup R^\prime \in y$ with representatives $i$ and $j$ of $R$ and $R^\prime$, respectively. Then:
        \begin{align*}
            \lambda(x,y)=
            \begin{cases}
                (0,\text{max}(i,j)) & \text{if $\mathcal{E}_{xy}$ coherent;}\\
                (2,\text{min}(i,j)) & \text{if $\mathcal{E}_{xy}$ is non-coherent.}
            \end{cases}
        \end{align*}
        \item Let $\mathcal{E}_{xy}$ be a signed edge. Then
        \[
            \lambda(x,y)=(1,1).
        \]
    \end{itemize}
    Labels are ordered lexicographically.
\end{df}

The next result follows from of \cite[Theorem 7]{DNG19} and \cite[Theorem 18]{DNG19}. 

\begin{thm}
    The labeling $\lambda$ of the definition above is an EL-labeling of $\Pi_n^B$.
\end{thm}

In the rest of this subsection we show that the above labeling is also an EL-labeling for the intersection lattices of $\D_n$- and $\D_{n,s}$-arrangements and examine their corresponding Chow polynomials in the spirit of \Cref{thm:chow_formula}.

We start by explicitly describing the intersection lattices of arrangements $\D_n$ and $\D_{n,s}$. According to \cite[Theorem 7]{DNG19} the intersection lattice $\latDn$ is isomorphic to the subposet of $\latBn$ consisting of all elements $x \in \latBn$ such that $|x_0|\not= 2$. In other words, we take  $\latBn$ and delete all partitions and their connecting edges that contain a block of the form $\{k,0,-k\}$ for $k \in [n]$.
From the bijection between $\latBn$ $\Pi_n^B$, it follows that the intersection lattice of $\D_{n,s}$ is isomorphic to the subposet of $\latBn$ consisting of all partitions of $\latBn$ without those that contain a block of the form $\{k,0,-k\}$ for $k \in \{s+1,\dots,n\}$.

The following theorem establishes that the EL-labeling defined in \Cref{def:concrete_el_labeling} is preserved when we go from $\latBn$ to $\latDns$. 

\begin{thm}
    \label{thm:EL_Dns}
    The EL-labeling $\lambda$ from Definition \ref{def:concrete_el_labeling} is an EL-labeling for the intersection lattice $\latDns $, for $s\in \{0,\dots,n\}$.
\end{thm}

\Cref{fig: EL_D_3} below illustrates the EL-labeling for the intersection lattice of the reflection arrangement $\D_3$.
In order to prove the above theorem we need the following three technical lemmata.

\begin{figure}[h]
            \centering
            \begin{tikzpicture}[scale=1,
            orangeline/.style={ draw=orange,line width = 0.75},
            cyanline/.style={draw=cyan, line width=0.75},
            violetline/.style={draw=violet, line width=0.75},
            greenline/.style={draw=green, line width=0.75},
            grayline/.style={draw=gray, line width=0.75}
            ]
            \node (1) at (0,0)  {\scriptsize $1|\sm 1|2|\sm2|3|\sm3$};
            \node (2) at (-5,2)  {\scriptsize$2 3|\sm 2 \sm 3|1|\sm 1$};
            \node (3) at (-3,2)  {\scriptsize$1 3|\sm 1 \sm 3|2|\sm 2$};
            \node (4) at (-1,2)  {\scriptsize$1 2|\sm 1 \sm 2|3|\sm 3$};
            \node (5) at (1,2)  {\scriptsize$2 \sm 3|\sm 2 3|1|\sm 1$};
            \node (6) at (3,2)  {\scriptsize$1 \sm 3|\sm 13|2|\sm 2$};
            \node (7) at (5,2)  {\scriptsize$1 \sm 2|\sm 1 2|3|\sm 3$};
            \node (8) at (-6,4)  {\scriptsize$123|\sm 1\sm2\sm3$};
            \node (9) at (-4,4)  {\scriptsize$1 \sm 2 \sm 3|\sm 1 2 3$};
            \node (10) at (-2,4)  {\scriptsize$230\sm 2 \sm 3|1|\sm 1$};
            \node (11) at (0,4)  {\scriptsize$130\sm1 \sm3|2|\sm2$};
            \node (12) at (2,4)  {\scriptsize$120\sm1\sm2|3|\sm3$};
            \node (13) at (4,4)  {\scriptsize$12\sm3|\sm1\sm2 3$};
            \node (14) at (6,4)  {\scriptsize$1 \sm2 3| \sm1 2 \sm 3$};
            \node (15) at (0,6)  {\scriptsize$1230\sm 1 \sm 2 \sm3$};

            \draw (1)--(2)[cyanline];
            \draw (1)--(3)[cyanline];
            \draw (1)--(4)[greenline];
            \draw (1)--(5)[grayline] ;
            \draw (1)--(6)[orangeline];
            \draw (1)--(7)[orangeline] ;
            
            \draw (2)--(8)[greenline] ;
            \draw (2)--(9)[orangeline];
            \draw (2)--(10)[violetline];

            \draw (3)--(8)[greenline];
            \draw (3)--(11)[violetline];
            \draw (3)--(14)[orangeline];

            \draw (4)--(8)[cyanline];
            \draw (4)--(12)[violetline];
            \draw (4)--(13)[orangeline];

            \draw (5)--(10)[violetline];
            \draw (5)--(13)[greenline];
            \draw (5)--(14)[orangeline];

            \draw (6)--(9)[orangeline];
            \draw (6)--(11)[violetline];
            \draw (6)--(13)[greenline];

            \draw (7)--(9)[orangeline];
            \draw (7)--(12)[violetline];
            \draw (7)--(14)[cyanline];

            \draw (8)--(15)[violetline];
            \draw (9)--(15)[violetline];
            \draw (10)--(15)[violetline];
            \draw (11)--(15)[violetline];
            \draw (12)--(15)[violetline];
            \draw (13)--(15)[violetline];
            \draw (14)--(15)[violetline];

            \begin{scope}[shift={(5.5,1.3)}]
                \draw(0,0)--(0.5,0)[greenline] node[right] {\scriptsize $(0,2)$};
                \draw(0,-0.35)--(0.5,-0.35)[cyanline] node[right] {\scriptsize $(0,3)$};
                \draw(0,-0.7)--(0.5,-0.7)[violetline] node[right] {\scriptsize $(1,1)$};
                \draw(0,-1.05)--(0.5,-1.05)[orangeline] node[right] {\scriptsize $(2,1)$};
                \draw(0,-1.4)--(0.5,-1.4)[grayline] node[right] {\scriptsize $(2,2)$};
            \end{scope}
            \end{tikzpicture}
        \caption{$\mathcal{L}(\D_3)$ with the EL-labeling.}
        \label{fig: EL_D_3}
    \end{figure}
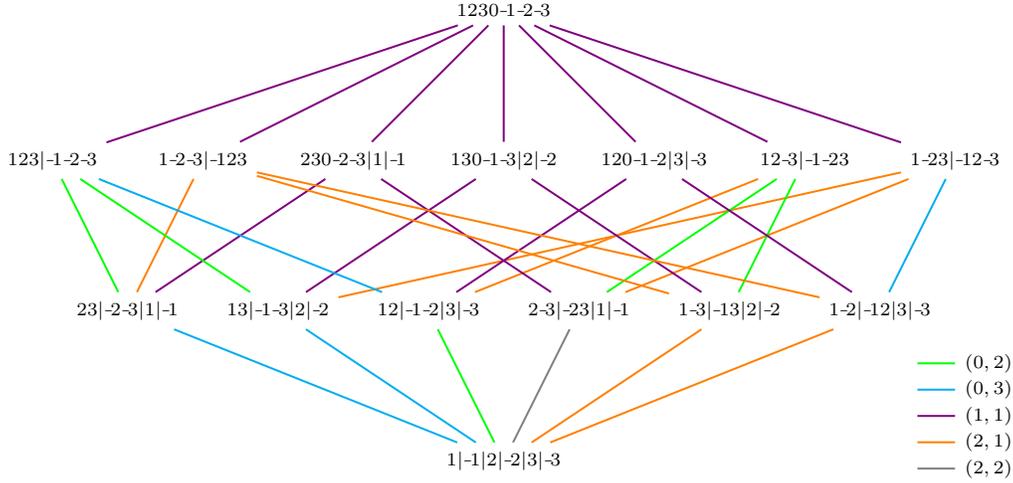
    
\begin{lem}
\label{lem:out}
Let $x \in \latDn$ and $y \in \latBn\setminus\latDn$ with $x \precdot y$. Then $\edge$ is a signed edge.
\end{lem}

\begin{proof}
    We know that $y \in \latBn\setminus\latDn$, therefore $y$ contains a block of the form $\{k,0,-k\}$, for some $k \in [n]$. At the same time, $x$ cannot contain such a block. Hence, the only possible way to go from $x$ to $y$ is via merging $\{k\}$ and $\{-k\}$ which creates the signed block $\{k,0,-k\}$. Thus $\edge$ is signed.
\end{proof}

\begin{lem}
\label{lem:in}
Let $x \in \latBn\setminus\latDn$ and $y \in \latDn$ with $x \precdot y$. Then $\edge$ is a signed edge.
\end{lem}

\begin{proof}
    We know that $x \in \latBn\setminus\latDn$, therefore $x$ contains a block of the form $\{k,0,-k\}$ for some $k \in [n]$. At the same time, $y$ cannot contain such a block. Hence, the only possible ways to go from $x$ to $y$ is via merging $\{k,0,-k\}$ with some blocks $A$ and $-A$ of $x$. This creates the signed block $A \cup -A \cup \{k,0,-k\}$  Thus $\edge$ is signed.
\end{proof}

\begin{lem}
\label{lem:Dns_noegde}
Let $1\leq u < s \leq n$. Moreover let $x \in \latDnu \setminus \latDn$ and $y \in \latDns \setminus \latDnu$ with either $\text{rk}(y) = \text{rk}(x)+1$ or $\text{rk}(x) = \text{rk}(y)+1$. Then in either case there is no edge between $x$ and $y$.
\end{lem}

\begin{proof}
    This follows by the incomparability statement of \Cref{prop:description-of-lattice} and the proof of \Cref{prop:arithmeticity-for-Dns}.
\end{proof}

\begin{proof}[Proof of \Cref{thm:EL_Dns}]
    Let $s \in \{0,\dots,n\}$. We know that the labeling from \Cref{def:concrete_el_labeling} is an EL-labeling for $\latBn$ and that $\latDns$ is a subposet of $\latBn$.~By \Cref{lem: subposet_criterion}, it is sufficient to show that for all $x \preceq y$ in $\latDns$ the unique increasing maximal chain in $[x,y] \subseteq \latBn$ is also contained in $\latDns$.
    Let $x,y \in \latDns$ with $x \preceq y$.~If  $[x,y]\subseteq\latDns$ we are done.
    For the other case, let $c$ be a maximal chain in $[x,y] \subseteq \latBn$ that is not fully contained in $\latDns$.
    We will show that $\lambda(c)$ contains the label $(1,1)$ at least twice, hence the chain cannot be increasing. 
    Now there exists an element $z \in \latBn\setminus\latDns$ such that $c = (x,\ldots,z) \cup (z,\ldots,y)$. 
    In particular, there are $x_1,z_1 \in (x,\ldots,z)$ such that $x_1 \in \latDn$, $z_1 \in \latBn\setminus\latDns$ and $x_1 \precdot z_1$, by \Cref{lem:Dns_noegde}. Now using \Cref{lem:out} it follows that $\mathcal{E}_{x_1 z_1}$ is a signed edge, hence $\lambda(x_1,z_1) = (1,1)$.
    Similarly, by \Cref{lem:Dns_noegde} there exist $z_2,y_1 \in (z,\ldots,y)$ such that $z_2 \in \latBn\setminus\latDns$, $y_1 \in \latDn$ and $z_2 \precdot y_1$. Using \Cref{lem:in} we get that $\mathcal{E}_{z_2 y_1}$ is a signed edge, therefore $\lambda(z_2,y_1) = (1,1)$. So the label $(1,1)$ appears at least two times. Consequently, the unique increasing maximal chain in $[x,y]$ is contained in $\latDns$.
\end{proof}

We can now prove the following interesting result:

\begin{lemma}\label{lemma:chains-and-union-additivity}
    Let $\L$ be a lattice, $\C$ a subset of its chains. We can consider elements of $\C$ inside different sublattices of $\L$. We let $\phi$ be a function on $\C$, independent on a sublattice we consider. Then
    $$
        f(\L') = \sum_{\mathcal F \in C,\, \mathcal{F}\subset \L'} \phi(\mathcal F)
    $$
    is union-additive on the sublattices of $\L$.
\end{lemma}
\begin{proof}
    Let $\L'$ be a sublattice of $\L$. If there is a coherent partition of $\L'\subset \L$ into $L_0$, $\dots$, $L_k$, each chain in $\mathcal C$ contained in $\L'$ is either entirely contained in $L_0$ or intersects exactly one $L_i$ for $i\in [k]$ by incomparability. 
\end{proof}
\begin{proof}[Proof of \Cref{prop:chow_arithmeticity_EL}via EL-labelling]
    If we fix a  some labelling $\lambda$ of a lattice $\L$ and $\mathcal F$ is a chain in $\L$,
    \[\phi(\mathcal F) = t^{\text{des}(\lambda_\mathcal{F})} (t+1)^{n-1-2\text{des}(\lambda_\mathcal{F})}, \]
    does not depend on a sublattice of $\L$ in which we consider the chain.
    We use \Cref{lemma:chains-and-union-additivity} with $\L = \L(\B_n)$, $\lambda$ being the EL-labelling above, $\C$ is the set of chains as in \Cref{thm:chow_formula} and $\phi$ as above.
    By Theorem \ref{thm:EL_Dns} and \Cref{thm:chow_formula}, $\sum_{\mathcal F\in \C,F\subset \L(\D_{n,s})} \phi(\mathcal F)$ is the Chow polynomial of $\D_{n,s}$, and by \Cref{lemma:chains-and-union-additivity} this is union-additive on sublattices of $\L(\B_n)$. So, \Cref{prop:arithmeticity-for-Dns} finishes the proof.
\end{proof}

\bigskip
\bigskip
\bigskip

\printbibliography
\newpage

\pagestyle{myheadings}
\markright{APPENDIX}
\appendix

\section{Small catalogue of Chow polynomials of $\D_{n,s}$- arrangements}\label{appendixsec: chow_polyn_catalogue}

We provide a small catalogue of Chow polynomials of matroids arising from $\D_{n,s}$ arrangements up to dimension $n=7$.

The Chow polynomials in below \Cref{table: chow_poly_234,table: chow_poly_56,table: chow_poly_7} were computed using the computer algebra system \verb|SageMath|, based on the previously introduced EL-labeling (\Cref{def:concrete_el_labeling}) and the formula from \Cref{thm:chow_formula}. The code used for these computations is available in the GitHub repository \cite{DHMPR25code}.

\begin{table}[h]
\hspace{-28pt}
    \begin{minipage}{.35\linewidth}
        \centering
        \ra{1.2}
        \hspace{50pt}
        \begin{tabular}{@{}ll@{}}
             \hline
             $s$& $H_{\D_{n,s}}(t)$ \\
             \hline
             0& $t+1$\\
             1& $t+1$\\
             2& $t+1$\\
             \hline
        \end{tabular}
    \end{minipage}
    \begin{minipage}{.3\linewidth}
        \centering
        \ra{1.2}
        \begin{tabular}{@{}ll@{}}
             \hline
             $s$& $H_{\D_{n,s}}(t)$ \\ 
             \hline
             0& $t^2+8t+1$\\
             1& $t^2 + 10t +1$\\
             2& $t^2 + 12t +1$\\
             3& $t^2 + 14t + 1$\\
         \hline
    \end{tabular}
    \end{minipage}
    \begin{minipage}{.3\linewidth}
        \centering
        \ra{1.2}
        \begin{tabular}{@{}ll@{}}
        \hline
             $s$& $H_{\D_{n,s}}(t)$ \\
             \hline
             0& $t^3 + 59t^2 + 59t + 1$\\
             1& $t^3 + 69t^2 + 69t +1$\\
             2& $t^3+79t^2+79t +1$\\
             3& $t^3 + 89t^2 + 89t+1$\\
             4& $t^3 + 99t^2 + 99t +1$\\
             \hline
        \end{tabular}
    \end{minipage}
    \caption{Chow polynomials for $n=2,3,4$.}
    \label{table: chow_poly_234}
\end{table}

\begin{table}[h]
    \begin{minipage}{.45\linewidth}
        \centering
        \ra{1.3}
        \begin{tabular}{@{}ll@{}}
        \hline
             $s$& $H_{\D_{n,s}}(t)$ \\ \hline
             0& $t^4 + 382t^3 + 1722t^2 + 382t +1$\\
             1&$t^4 + 430t^3 + 2010t^2 + 430t +1$\\
             2& $t^4 + 478t^3 + 2298t^2 + 478t +1$\\
             3& $t^4 + 526t^3 + 2586t^2 + 526t +1$\\
             4& $t^4 + 574t^3 + 2874t^2 + 574t +1$\\
             5& $t^4 + 622t^3 + 3162t^2 + 622t +1$\\
             \hline
        \end{tabular}
    \end{minipage}
    \begin{minipage}{.45\linewidth}
        \centering
        \ra{1.3}
        \begin{tabular}{@{}ll@{}}
        \hline
             $s$& $H_{\D_{n,s}}(t)$ \\ \hline
             0& $t^5 + 2515t^4 + 35956t^3 +35956t^2 + 2515t +1$\\
             1&$t^5 + 2771t^4 + 40932t^3 +40932t^2 + 2771t +1$\\
             2& $t^5 + 3027t^4 + 45908t^3 +45908t^2 + 3027t +1$\\
             3& $t^5 + 3283t^4 + 50884t^3 +50884t^2 + 3283t +1$\\
             4& $t^5 + 3539t^4 + 55860t^3 +55860t^2 + 3539t +1$\\
             5& $t^5 + 3795t^4 + 60836t^3 +60836t^2 + 3795t +1$\\
             6& $t^5 + 4051t^4 + 65812t^3 +65812t^2 + 4051t +1$\\
             \hline
        \end{tabular}
    \end{minipage}
    \caption{Chow polynomials for $n=5,6$.}
    \label{table: chow_poly_56}
\end{table}

\begin{center}
\begin{table}[h]
    \centering
    \ra{1.3}
    \begin{tabular}{@{}ll@{}}
    \hline
         $s$& $H_{\D_{n,s}}(t)$ \\ \hline
         0& $t^6 + 17824t^5 + 665107t^4 + 1980008t^3 + 665107t^2 + 17824t + 1$\\
         1&$t^6 + 19362t^5 + 742263t^4 + 2229164t^3 + 742263t^2 + 19362t + 1$\\
         2&$t^6 + 20900t^5 + 819419t^4 + 2478320t^3 + 819419t^2 + 20900t + 1$\\
         3&$t^6 + 22438t^5 + 896575t^4 + 2727476t^3 + 896575t^2 + 22438t + 1$\\
         4&$t^6 + 23976t^5 + 973731t^4 + 2976632t^3 + 973731t^2 + 23976t + 1$\\
         5&$t^6 + 25514t^5 + 1050887t^4 + 3225788t^3 + 1050887t^2 + 25514t + 1$\\
         6& $t^6 + 27052t^5 + 1128043t^4 + 3474944t^3 + 1128043t^2 + 27052t + 1$\\
         7&$t^6 + 28590t^5 + 1205199t^4 + 3724100t^3 + 1205199t^2 + 28590t + 1$\\
         \hline
    \end{tabular}
    \caption{Chow polynomials for $n=7$}
    \label{table: chow_poly_7}
\end{table}
\end{center}

\end{document}